\documentclass[12pt]{article}
\usepackage{array,ifthen,amsmath,amssymb,amsthm,latexsym,color,epsfig,bm,verbatim,hyperref,enumerate}
\usepackage[margin = .8in]{geometry}
\hypersetup{colorlinks=true,linkcolor=blue,citecolor=blue,filecolor=blue,urlcolor=blue}

\newcommand{\cF}{{\mathcal{F}} }

\newcommand{\cC}{{\mathcal{C}} }
\newcommand{\cK}{{\mathcal{K}} }
\newcommand{\wX}{{\widehat{X}} }
\newcommand{\wY}{{\widehat{Y}} }
\newcommand{\tX}{{\widetilde{X}}}
\newcommand{\tY}{{\widetilde{Y}}}

\newcommand{\eee}{{\mathrm{e}}}   
\newcommand{\ind}[1]{{\mathbf{1}\!\left[{#1}\right]}}

\newcommand{\Expectation}{\mathrm{E}}
\newcommand{\Exp}[1]{{\Expectation\left[{#1}\right]}}

\newcommand{\ExpCond}[2]{{\Expectation\left[{#1} \mid {#2} \right]}}
\newcommand{\ExpBCond}[2]{{\Expectation\left[{#1} \Big\vert {#2} \right]}}

\newcommand{\Probability}{\mathrm{Pr}}
\newcommand{\Prb}[1]{{\Probability\left[{#1}\right]}}

\newcommand{\PrbCond}[2]{{\Probability\left[{#1} \mid {#2} \right]}}
\newcommand{\PrbBigCond}[2]{{\Probability\left[{#1} \, {\bf \big\vert} \, {#2} \right]}}

\newcommand{\dtv}[2]{d_{\mathrm{TV}}(#1,#2)}

\newcommand{\Tmix}{T_{\mathrm{mix}}}

\newcommand{\Aval}{{\cal A}}

\def\cH{{\cal F}}

\def\cU{{\cal U}}

\def\cH{{\cal F}}

\def\cP{{\cal P}}

\def\D{\Delta}
\def\e{\epsilon}

\def\r{\rho}

\def\OM{\Omega}

\newcommand{\brac}[1]{\left( #1\right)}

\def\cH{{\cal H}}
\def\cG{{\cal G}}

\newcommand{\eps}{\epsilon}

\newcommand{\refer}[1]{(\ref{#1})}

\newtheorem{theorem}{Theorem}[section]
\newtheorem{lemma}[theorem]{Lemma}

\newtheorem{corollary}[theorem]{Corollary}

\title{Randomly coloring planar graphs with fewer colors than the maximum degree \footnotetext{An extended abstract of this paper appeared in {\em
Proceedings of the 39th Annual ACM Symposium on Theory of Computing} (STOC),
450-458, 2007.  This version contains complete and significantly revised proofs.}
}
\author{
Thomas P. Hayes\thanks{Department of Computer Science, University of New Mexico,
Albuquerque, NM 87131.
Email: hayes@cs.unm.edu.
}
\and
Juan C. Vera\thanks{Department of Econometrics and Operations Research,
Tilburg University,
5000 LE Tilburg,
The Netherlands.
Email: j.c.veralizcano@uvt.nl.
}
\and
Eric Vigoda\thanks{College of Computing, Georgia
Institute of Technology, Atlanta GA 30332.  Email: vigoda@cc.gatech.edu.
Research supported in part by NSF grants CCF-0830298 and CCF-0910584.}
}
\begin{document}




\pagenumbering{arabic}

\maketitle

\begin{abstract}
We study Markov chains for randomly sampling $k$-colorings
of a graph with maximum degree $\Delta$.
Our main result is a polynomial upper bound on the
mixing time of the single-site update chain known
as the Glauber dynamics for planar graphs when $k=\Omega(\Delta/\log{\Delta})$.
Our results can be partially extended to the more general
case where
the maximum eigenvalue of the adjacency matrix of the graph
is at most  $\Delta^{1-\eps}$, for fixed $\eps > 0$.

The main challenge when $k \le \Delta + 1$ is the possibility of
``frozen'' vertices, that is, vertices for which only one color
is possible, conditioned on the colors of its neighbors.
Indeed, when $\Delta = O(1)$, even a typical coloring can
have a constant fraction of the vertices frozen.
Our proofs rely on recent advances in techniques
for bounding mixing time using ``local uniformity'' properties.
\end{abstract}

\newpage

\section{Introduction}

Markov chains for randomly sampling (and approximately counting) $k$-colorings of an input graph have been studied
intensively in recent years.  The colorings problem is appealing as a natural combinatorial problem,
as a noteworthy example of a \#P-complete problem,
and as a challenging example of the general class of spin systems from statistical physics, which
includes problems such as the independent sets (or hard-core model) and Ising model.
Improved results for sampling/counting colorings have been in lock-step with
advances in the use of coupling techniques.
The study of the convergence rate of Markov chains for spin systems has close intuitive (and some formal)
connections with macroscopic properties of corresponding statistical physics models.

Considerable attention has been paid to the Glauber dynamics,
which is of particular interest for its simplicity and intimate connections to properties of infinite-volume
Gibbs distributions (e.g., see \cite{Weitz,DSVW,Martinelli}).  In the
(heat-bath) Glauber dynamics, at each step, a random
vertex is recolored with a color chosen randomly from those colors not appearing in its
neighborhood.  For a graph with maximum degree $\Delta$,
when $k\geq \Delta+2$ the Glauber dynamics is ergodic with unique stationary
distribution uniform over the $k$-colorings of $G$.
The mixing time of the Glauber dynamics is the number of steps, from the worst
initial state, to get within (total) variation distance $\leq 1/4$ of the stationary distribution.

A large body of work has studied the following folklore conjecture:  For an input graph with
maximum degree $\Delta$, the Glauber dynamics has $O(n\log{n})$ mixing time
whenever $k\geq\Delta+2$.
Such a mixing time is optimal, as shown by Hayes and Sinclair \cite{HS}, and leads to
a fully-polynomial randomized approximation scheme for counting $k$-colorings for any
$k\geq\Delta+2$.
For general graphs, $\Delta+2$ is a clear  lower bound since
there exist graphs where the Glauber dynamics is not
ergodic below this threshold (and some graphs are not colorable below $\Delta+1$).
We will prove optimal mixing of the Glauber dynamics for $k<<\Delta$ for a large
class of graphs, including all planar graphs.

Martinelli, Sinclair and Weitz \cite{MSW} proved
$O(n\log{n})$ mixing time of the Glauber dynamics when $k\geq\Delta+2$
for the complete $(\D-1)$-ary tree with arbitrary boundary conditions
(that is, a fixed coloring of the leaves).
Their result is optimal for worst-case boundary conditions
since below $\Delta+2$ some boundary conditions can
``freeze'' the entire tree.
For graphs with sufficiently large girth $g>10$
and large maximum degree $\Delta=\Omega(\log{n})$, Hayes and Vigoda
proved $O(n\log{n})$ mixing time when $k\geq(1+\eps)\Delta$, for any $\eps>0$.
Their work built upon upon a long series of earlier works
(see \cite{FriezeVigoda} for a survey),
and still seems far from addressing the conjecture without additional
girth and degree assumptions.
Recently, Hayes \cite{Hayes} presented a relatively simple proof of $O(n\log{n})$
mixing time of the
Glauber dynamics for any planar graph when $k\geq\Delta+O(\sqrt{\Delta})$.

The $k\geq\Delta+2$ threshold is a natural threshold from a statistical physics perspective.
On the infinite $(\D-1)$-ary tree, $\Delta+2$ is the threshold for the persistence of long-range
interactions, more precisely, uniqueness/non-uniqueness of infinite-volume Gibbs measures
\cite{Jonasson,BrightwellWinkler}.
More precisely, when $k<\Delta+2$ a fixed coloring
of the leaves influences the coloring of the
root.  In fact, some colorings of the leaves
``freeze'' the coloring for the remainder of the tree.
The existence of frozen colorings on the tree
when $k<\Delta+2$ hints at the major obstacle
we need to overcome to prove rapid mixing when $k<<\Delta$.

In this paper, we get below the $\Delta+2$
threshold for trees and for
all planar graphs.   Our results suggest
that for planar graphs the threshold for
rapid mixing of the Glauber dynamics is $k = \Theta(\Delta/\log{\Delta})$.
Note that, even for planar graphs,
$\Delta/\log{\Delta}$ cannot be
replaced by a smaller power of $\Delta$,
since on any tree of maximum degree $\Delta$,
an easy conductance argument shows that
the Glauber dynamics has mixing time
$\Omega(n \exp(\Delta/k))$,
which is superpolynomial in $n$ when
$k = o(\Delta/\log n)$.
The only previous rapid mixing results for $k<\Delta$ were for
3-colorings of finite subregions of the 2-dimensional integer lattice \cite{LRS,GMP},
and random graphs~\cite{DFFV} (subsequent to the initial publication
of this work, Mossel and Sly \cite{Sly-Mossel} presented improved results on sparse random graphs).

Our work builds upon the ideas of Hayes \cite{Hayes}
to utilize small operator norm $\rho$.
(The operator norm, or ``spectral radius,'' equals the maximum eigenvalue
of the adjacency matrix of the graph.)
In addition to the spectral properties, an important component of our work is
proving ``local uniformity'' properties for graphs with small
operator norm.  For example, showing that for a random coloring,
the colors appearing in the neighborhood of a vertex are roughly independent.
Such local uniformity properties have been the basis for many
previous results for colorings, beginning with Dyer and Frieze \cite{DyerFrieze}
(see \cite{FriezeVigoda} for a survey).
The challenging aspect in our work is that since $k<<\Delta$, there
are nearly frozen colorings, hence even ergodicity is not obvious.
For graphs with large maximum degree we prove
that the local uniformity properties hold with high probability, building upon \cite{FriezeVera}.
This leads to the following theorem, whose proof uses the coupling with stationarity approach
of \cite{HV-CwS}.

\begin{theorem}
\label{thm:high-degree}
For all $\eps>0$, for all $G$ with operator norm
$\leq \Delta^{1-\eps}/2$ and $\Delta=\Omega(\log^{1+\e}{n})$, all $k>4\eps^{-1}\Delta/\ln{\Delta}$,
the Glauber dynamics has mixing time $O(n\log{n})$.
\end{theorem}

Removing the degree restriction presents major obstacles since for a random
coloring, a constant fraction of the vertices are frozen.
We introduce a new Markov chain which is a more natural chain to
both implement and analyze for graphs with operator norm $\leq\Delta^{1-\eps},\eps>0$.
It is a generalization of the standard dynamics for bipartite graphs in
which we alternately recolor all of the vertices in one of the two partitions.
We refer to the new chain as the {\em level-set dynamics}.  We partition the
vertices into level sets and then successively recolor the sets.
In Section \ref{sec:level-sets} we present our partition of the vertices
into level sets $L_0,\dots,L_h$ based on the principal eigenvector
of the adjacency matrix.

Consider a partition of the vertices $V=L_0\cup L_1\cup \dots \cup L_{h}$.
One {\em scan} of the graph $G$ by the level-set dynamics has rounds $j=0,\dots,h$,
where in round $j$ we do $|L_j|\log{\Delta}$
random recolorings (i.e., Glauber updates) of vertices in $L_j$.
(If the set $L_j$ is an independent set, then we can
instead simply recolor the vertices of $L_j$ once
in arbitrary order.)
We define the mixing time of the level-set dynamics as the the number of scans
until we are within variation distance $\leq 1/4$ of the uniform distribution.
The level-set dynamics can be viewed as a common generalization
of the Glauber dynamics (corresponding to the partition $L_0=V$)
and the systematic scan dynamics (corresponding to 
the partition into singletons).
Systematic scan is popular in experimental work, but often appears more
difficult to analyze than the Glauber dynamics, see, e.g.,~\cite{DGJ}.

We use the level-set dynamics where the vertices are partitioned
into level sets based on their entry in the principal eigenvector.
We formally define our partition into level sets in Section \ref{sec:level-sets}.
We now formally state the main theorem of this paper.

\begin{theorem}
\label{thm:constant-degree}
There exists $\Delta_0>0$ such that for every planar graph $G$ of maximum degree $\Delta>\Delta_0$, for $k > 100 \Delta/\log \Delta$ colors the following hold:
\begin{enumerate}
\item[(i)]
We can compute in $O(n^3)$ time a partition of $V$ into $h=O(\log{n})$ sets such that
the level-set dynamics mixes within $O(\log{n})$ scans of $G$ (and thus a total of $O(n\log{n}\log{\Delta})$ Glauber steps)
 and
\item[(ii)]
The Glauber dynamics has mixing time $O(n^3\log^9{n})$.
\end{enumerate}
\end{theorem}

The polynomial mixing time of the Glauber dynamics follows
from the above result for the level-set dynamics with
a straightforward comparison argument.
For completeness we include the comparison proof in Section \ref{sec:comparison}.
The proof of Theorem \ref{thm:constant-degree} uses ideas presented in
\cite{DFHV} for utilizing local uniformity properties for constant degree graphs.

Although uniqueness of the infinite-volume
Gibbs measure may be the key concept for rapid mixing of the Glauber dynamics
on general graphs,
our results show that, at least in the case of planar graphs,
there is a second threshold for rapid mixing.
This threshold may correspond to extremality
of the free measure (that is, no boundary condition)
in the set of infinite-volume Gibbs measures.
Subsequent to the initial publication of this work, the threshold
for extremality of the free measure in the tree was established at $k=(\Delta/\ln{\Delta})(1+o(1))$ \cite{BVVW,Sly}.  More recent work of Tetali et al \cite{TVVY} shows that the mixing time
of the Glauber dynamics on the complete tree undergoes a 
phase transition at (up to first order terms) the
reconstruction threshold.

In the following section, we present some basic foundational results relevant to this work,
including ergodicity of the Glauber dynamics in our setting, the definition of the level sets,
and basic properties of the level sets.
In Section \ref{sec:high} we prove Theorem \ref{thm:high-degree} for high degree graphs.
Our main result, Part (i) of Theorem \ref{thm:constant-degree}, is proved
in Sections \ref{sec:constant-degree} and \ref{sec:tech-proofs}.
 The coupling proof
is done in Section \ref{sec:constant-degree}, and in Section \ref{sec:tech-proofs} we prove
the local uniformity properties of the Glauber dynamics for constant degree planar graphs.
In Section \ref{sec:comparison} we prove part (ii) of Theorem \ref{thm:constant-degree}.

\section{Preliminaries}
\label{sec:preliminaries}

\subsection{Basic Notation}

We begin by specifying some notation which will be used
throughout the paper.  Let $G = (V,E)$ be the graph to
be colored, and let $k$ denote the number of colors
to be used.  We say a function $f: V \to \{1,...,k\}$
is a \emph{proper $k$-coloring} of $V$ if, for every
edge $\{u,v\} \in E$, $f(u) \ne f(v)$.
Let $\Omega$ denote the set of all proper $k$-colorings
of $G$.
For $X \in \OM$ and $v \in V$, let
\[ \Aval_X(v) = [k] \setminus X(N(v))
\] be the set of available colors for $v$ in $X$.

\subsection{Mixing Time}

For a pair of distributions $\mu$ and $\nu$ on a finite space $\Omega$,
their (total) variation distance is defined to be:
\[
\dtv{\mu}{\nu}  := \frac{1}{2} \sum_{x\in\Omega} |\mu(x) - \nu(x)|.
\]
Let $\Omega$ denote the set of proper $k$-colorings of the input graph $G$.
Let $\pi$ denote the uniform distribution over $\Omega$.
For an ergodic Markov chain on state space $\Omega$
with unique stationary distribution $\pi$
and transition matrix $P$, its mixing time $\Tmix$ is defined as
\[  \Tmix = \max_{X_0\in\Omega} \min\{t: \dtv{P^t(X_t,\cdot)}{\pi} \leq 1/4\}.
\]

\subsection{Coupling and Disagreement Percolation}
\label{sec:coupling}

We use the coupling method to bound the mixing time.  For an introduction
to the coupling method see Jerrum \cite{Jerrum:book} or Levin, Peres and Wilmer \cite{LPW}.
We will define a coupling for two copies $(X_t)$ and $(Y_t)$ of the dynamics under
consideration.  The coupling inequality \cite{Aldous} says that if there is a time $T$
and for every pair of pair of initial
states $X_0,Y_0$ there is a coupling such that $\PrbCond{X_T\neq Y_T}{X_0,Y_0}\leq 1/4$,
then $\Tmix\leq T$.

\subsection*{Jerrum's Coupling}

We will use the same coupling as studied by Jerrum \cite{Jerrum}.
For a planar graph $G$, for initial colorings $X_{0,0}$ and $Y_{0,0}$, for
$0\leq i \leq h, 0\leq t< T_i$, let $X_{i,t}$ and $Y_{i,t}$
denote two copies of the level-set dynamics in step $t$ within round $i$.
Given $X_{i,t}$ and $Y_{i,t}$ the coupling chooses the same vertex
$v$ in level $L_i$ to update in both chains.  Then we
couple the color choice for $v$ in the two chains so as to maximize
the probability that $v$ receives the same color in both chains.
More precisely, given $X_{i,t}$ and $Y_{i,t}$, we define $(X_{i,t+1},Y_{i,t+1})$
as follows:
\begin{enumerate}
\item \label{step:choose-v}
Choose $v$ uniformly at random from $L_i$.
\item For all vertices $w\neq v$, set $X_{i,t+1}(w)  = X_{i,t}(w)$
and  $Y_{i,t+1}(w)  = Y_{i,t}(w)$.
\item Without loss of generality, assume that $|\Aval_{X_{i,t}}(v)|\leq |\Aval_{Y_{i,t}}(v)|$
(otherwise, interchange the roles of $X$ and $Y$ in the below algorithm).
\item
Let $C = \Aval_{X_{i,t}}(v)\bigcap \Aval_{Y_{i,t}}(v)$ be the common available colors,
and denote the disagreeing colors by
$D_X = \{d_1,\dots,d_j\} =  \Aval_{X_{i,t}}(v)\setminus \Aval_{Y_{i,t}}(v)$,
and $D_Y =  \{d'_1,\dots,d'_\ell \} = \Aval_{Y_{i,t}}(v)\setminus \Aval_{X_{i,t}}(v)$.
\item
Choose $c_Y$ uniformly at random from $\Aval_{Y_{i,t}}(v)$.
\item If $c_Y\in C$, then set $X_{i,t+1}(v) = Y_{i,t+1}(v) = c_Y$.
\item If $c_Y=d'_m$ for $m\leq j$ then set $X_{i,t+1}(v) = d_m$ and $Y_{i,t+1}(v) = c_Y=d'_m$.
\item If $c_Y=d'_m$ for $m>j$, then first we choose $c_X$ uniformly at random
from $X_{i,t+1}(v)$, and, finally, set
$X_{i,t+1}(v) = c_X$ and
$Y_{i,t+1}(v) = c_Y=d'_m$.
\end{enumerate}
Roughly, the coupling for the color choice for $X_{i,t+1}(v)$ and $Y_{i,t+1}(v)$ works
by, for each $c\in C$, setting $X_{i,t+1}(v)=Y_{i,t+1}(v) = c$ with probability
$1/\max\{|\Aval_{X_{i,t}}(v)|,|\Aval_{Y_{i,t}}(v)|\}$, and with the remaining probabilities
choosing from the respective distributions over $\Aval_{X_{i,t}}(v)$ and $\Aval_{Y_{i,t}}(v)$ so that $X_{i,t+1}(v)$ (and similarly $Y_{i,t+1}$) is uniformly distributed over $\Aval_{X_{i,t}}(v)$ (over $\Aval_{Y_{i,t}}(v)$).

\subsection{Operator norm}
\label{sec:spec-radius}

Theorem \ref{thm:high-degree}
applies to graphs with $\rho\leq \Delta^{1-\eps}$ for any $\eps>0$.
Examples of such graphs are the following (e.g., see \cite{Biggs}):
\begin{itemize}
\item
Planar graphs, which have $\rho\leq 2\sqrt{3(\Delta-3)}$ (c.f., \cite[Corollary 17]{Hayes} or \cite{DvorakMohar} for recent improvements).
\item
Graphs embeddable on any fixed surface of finite genus.
\item
Bipartite graph $G=(V_1\cup V_2,E)$
has $\rho=\sqrt{\Delta_1\Delta_2}$ where
$\Delta_i$ is the maximum degree of vertices in $V_i$, $i=1,2$.
Thus bipartite graphs where one side of the bipartition has maximum degree $\Delta^{1-2\eps}$
satisfy the assumptions of our theorem.
\item Generalizing the previous example,
any graph such that
the product of degrees of any two adjacent vertices
is at most $\Delta^{2 - 2\eps}$.
\item Unions of any fixed number of the above,
since the operator norm is subadditive.
\end{itemize}

We now point out
that, for graphs with small operator norm
the Glauber dynamics is ergodic with many fewer colors than the
maximum degree.

\subsection{Upper bounds on diameter of the Glauber dynamics}
\label{sec:diameter}

In general, when $k \le \Delta + 1$, it is possible
that the Glauber dynamics is not connected; for
example, when $G$ is the complete graph on $n = \Delta + 1$
vertices.  In this case, using $k = \Delta + 1$ colors,
every coloring is ``frozen,'' meaning that the
connected components of $\Omega$ under the Glauber
dynamics are all singletons.  However, we
restrict our attention to a ``nicer''
class of graphs, for which we will see that fewer
colors are needed.

We now derive some fairly straightforward bounds
on the diameter of $\Omega$ in terms of the
max-min degree over subgraphs of $G$.

\begin{theorem}\label{thm:diameter-bd-n^2}
Suppose every subgraph of $G$ contains at least
one vertex of degree $\le d$.  Then for every
$k \ge 2(d+1)$, the Glauber dynamics is ergodic.
Indeed, the diameter of $\Omega$ with respect to
Glauber dynamics is at most $n^2-n$.
\end{theorem}

\begin{proof}
Inductively order $V = \{v_1, \dots, v_n\}$ so that,
for every $i$, $v_i$ has at most $d$ neighbors
among $v_{i+1}, \dots, v_n$.  This can be done greedily,
using the definition of $d$.  Observe that $G$ can
now be $(d+1)$-colored by simply greedily assigning
legal colors to the vertices in the order
$v_n, v_{n-1}, \dots, v_1$.  Fix such a $(d+1)$-coloring
$X$.  To walk from an arbitrary $Y \in \Omega$
to $X$, we will proceed in $n$ rounds, as follows.

In round $i \ge 1$, recolor vertices $v_i, v_{i-1}, \dots, v_1$
in that order.  When $i$ has the same parity as $n$,
use only colors from the set $\{1, \dots, d+1\}$.
When $i$ has the opposite parity from $n$, use
colors from the set $\{d+2, \dots, 2d+2\}$.
Note that there always is such a color available
at each step, since whenever recoloring a vertex $v_j$,
at most $d$ colors are forbidden due to neighbors
$v_i$ where $i > j$, and no colors from the allowed
set of colors (either $\{1, \dots, d+1\}$ or
$\{d+2, \dots, 2d+2\}$) are present among the vertices
$v_i$ where $i \le j$, since all were recolored
using the other color set in the previous round.
In the final round $n$, choose the colors to agree
with $X$, instead of arbitrarily; since none
of these colors are in use at the beginning
of round $n$, there is nothing to prevent this.

Combining this with the triangle inequality shows
that the diameter of $\Omega$ is at most
$2\binom{n}{2} = n^2-n$.
\end{proof}

This yields the following two corollaries.

\begin{corollary}\label{cor:planar-ergodic}
For planar graphs, when $k \ge 12$, the
diameter of $\Omega$ is at most $n^2 - n$.
\end{corollary}

\begin{proof}
We use the fact that a planar
graph on $n$ vertices has average degree
at most $6(1-2/n)$, which is in turn a
consequence of Euler's formula.
The result now follows by Theorem~\ref{thm:diameter-bd-n^2}
\end{proof}

\begin{corollary}\label{cor:spectral-ergodic}
For a graph with spectral radius $\rho$,
when $k \ge 2(\rho + 1)$, the diameter of
$\Omega$ is at most $n^2 - n$.
\end{corollary}

\begin{proof}
Let $\delta=\delta(G)$ denote the maximum, over subgraphs $H$ of $G$,
of the average vertex degree in $H$.  Note,
\begin{equation}
\label{eq:delta-rho}
\delta \le \rho
\end{equation}
To see this, let $e_U$
the characteristic  vector of $U$ in $V$.
Then, $|E(U)| = \sum_{u\in U}|N(u) \cap U| = e_U^T A e_U
\leq \rho(G) \|e_U\|^2 = \rho(G) |U|$.
Thus, the average degree in $U$ is $\leq \rho$ which implies
$\delta\leq\rho$.
The lemma now follows from \eqref{eq:delta-rho}
with Theorem~\ref{thm:diameter-bd-n^2}.
\end{proof}

\section{Level Sets}
\label{sec:level-sets}

An important component of our work is the partition of $G$ into level sets
based on the principal eigenvector of the graph.
Let $\r$ be the operator norm of $A$, an adjacency matrix of $G$.
Let $J$ denote the $n \times n$ all-ones matrix.
Note that the perturbed adjacency matrix $\widetilde{A} := A + \frac \r n J$
has maximum eigenvalue $\widetilde{\rho}$ which satisfies
$\rho < \widetilde{\rho} \le 2\rho$.

Let $w \in R^n_+$ be an eigenvector of $\widetilde{A}$ such that
$\widetilde{A}w = \widetilde{\r} w$.  Note that this implies
that all entries of $w$ are strictly positive.  Moreover,
$\widetilde{\r} w_{v} \geq \frac \r n \|w\|_1$ for all $v \in V$.
In particular, if we let $w_{\min}$ denote the minimum entry of $w$, then
$w_{\min} \geq \frac 1{2n} \|w\|_1$.
Our proofs will analyze a coupling argument
 using a weighted Hamming distance defined using this eigenvector $w$.
For every $S \subseteq V$ denote $w(S) := \sum_{s\in S}w(s)$.
Notice that for every $u\in V$,
\begin{equation}
\label{eq:neighbor-weight}
w(N(u)) = (A w)(u) \le (\widetilde{A} w)(u) \le \widetilde{\r} w(u).
\end{equation}

Our level-set dynamics for Theorem \ref{thm:constant-degree} uses level sets
defined by $w$.
Let $\eps>0$ be such that $\rho\le\Delta^{1-\eps}/2$ and hence
$\widetilde{\rho}\leq\Delta^{1-\eps}$.
We define the {\it level sets}:
\[
L_i = \left\{ v : \Delta^{i\eps/2} \leq \frac{w(v)}{w_{\min}} < \Delta^{(i+1)\eps/2} \right\}.
\]
Let $L_{<i} = \bigcup_{j<i} L_j$ and $L_{[i,j]}=\bigcup_{\alpha=i}^j L_\alpha$.
The level sets are also used for the uniformity results needed
in the proof of both theorems.

\subsection{Basic Properties of the Level Sets}
\label{sec:properties}

We first define a bound showing that most neighbors of a vertex lie in lower levels.
For $v\in L_i$, since $\widetilde{\rho}\le\Delta^{1-\eps}$,
 by \eqref{eq:neighbor-weight} and
the definition of $L_i$ we have:
\[  w(N(v)) \leq \widetilde{\rho} w(v) \leq \D^{1-\eps}\D^{\eps (i+1)/2}.
\]
Also,
\[  w(N(v)) \geq w(N(v)\setminus L_{<i}) \geq |N(v)\setminus L_{< i }|\D^{\eps i/2}.
\]
Hence,  for $v\in L_i$,
\begin{equation}
\label{eq:expansion-up}
  |N(v)\setminus L_{<i}| < \D^{1-\eps/2}.
\end{equation}

We can also bound the maximum number of levels traversed by an edge.
Let
\[  M = \max_{\{u,v\}\in E(G) }|\ell(u) - \ell(v)|
\]
where $u\in L_{\ell(u)}$ and $v\in L_{\ell(v)}$.
Since $\Delta^{i\eps/2} > \Delta^{1-\eps}$ for $i>(2/\eps) - 2$, we have that
\begin{equation}
\label{eq:bound-M}
M\leq (2/\eps)-2.
\end{equation}
This bound will be important in our proof of Theorem \ref{thm:constant-degree}
for planar graphs.  For planar graphs, we have that $\eps\geq 9/20$ for $\Delta$
sufficiently large.  Hence, $M\le 2$ for planar graphs.

We now bound the total number of levels.
Let $h+1$
denote the number of levels,
and let $w_{\max}$ denote the maximum entry of $w$.
Note,
\begin{equation}\label{eq:wminbound}
w_{\min} \geq \frac{1}{2n}\|w\|_1\geq \frac{1}{2n}w_{\max}.
\end{equation}
Hence,
\begin{equation} \label{eq.bd-m}
h \leq \frac{2}{\eps}\frac{\ln(2n)}{\ln{\Delta}}.
\end{equation}

\section{Graphs of large degree}
\label{sec:high}

In this section, we prove Theorem \ref{thm:high-degree}
via a coupling argument.
Consider two copies of the Glauber dynamics,
$(X_t,\, t\geq 0)$ and $(Y_t,\, t\geq 0)$.
Define the set of disagreements at time $t$ as
\[
D_t =\{v\in V:X(v) \neq Y(v)\}.
\]

We couple the two processes using Jerrum's coupling \cite{Jerrum}
that we defined in Section \ref{sec:coupling}.
Let $\eps>0$ be such that $\r \leq \Delta^{1-\eps}/2$
and $k> 4\eps^{-1}\Delta/\ln{\Delta}$.
We will prove that if $\D = \OM(\ln^{1+\e}n)$,
under Jerrum's coupling, for any $X_0,Y_0$, for $T=O(n\log{n})$ we have that:
\begin{equation}
\label{eq:contract-high-deg}
  \ExpCond{w(D_T)}{X_0,Y_0} \leq w_{min}/4
  \end{equation}
This implies Theorem \ref{thm:high-degree} in the following manner:
  \[   \PrbCond{X_{T}  \neq Y_{T}}{ X_0,  Y_{0}}
    =
\PrbCond{D_T \geq 1}{ X_{0},  Y_{0}}
\leq
  \ExpBCond{   \left| D_{T}\right| }{ X_0,  Y_{0}  }
\leq
  \ExpBCond{   \frac{  w(D_T) }{  w_{\min}  } }{ X_0,  Y_{0}   }
\leq
1/4.
\]
By the coupling inequality (see Section \ref{sec:coupling}), this proves that after $T=O(\log{n})$ steps, the Glauber dynamics is within variation distance $\le 1/4$ of the stationary distribution, and hence $\Tmix\leq T$.

It remains to prove \eqref{eq:contract-high-deg}.
For $Y \in \OM$ and $v \in V$, recall $\Aval_Y(v) = [k] \setminus Y(N(v))$.

For all $t \geq 0$, given $X_t,Y_t$, we have
\begin{align}
\nonumber
\Exp{w(D_{t+1})|X_t,Y_t} - w(D_t)
&=  \frac 1n\sum_{v\in V}w(v)\PrbCond{v\in D_{t+1}}{X_t,Y_t,\,v \text{ chosen at time } t}
 - \frac 1n\sum_{v\in D_t}w(v)\\
\label{keyIdea}
&\le
\frac 1n\sum_{v\in V}w(v)\frac {|N(v)\cap D_t|}{|\Aval_{Y_t}(v)|} - \frac 1nw(D_t)
\end{align}

The key to the proof of the theorem will be the following local uniformity property.
\begin{lemma}\label{lem:unif-HD} Let $\eps>0$ be given.
Let $G$ be a graph such that $\r \leq \Delta^{1-\eps}/2$
and $\Delta=\Omega(\ln^{1+\e}{n})$, and let $k>4\eps^{-1}\Delta/\ln{\Delta}$.
Let $Y$ be chosen uniformly from $\OM$, the set of all proper $k$-colorings of $G$. Then,
\begin{equation} \label{eq.unif-HD}
\Prb{\exists v\in V, \ |\Aval_Y(v)| < \D^{1-\e/2} } \leq n^{-4}.
\end{equation}
\end{lemma}
The lemma is related to uniformity properties originally used by Dyer and Frieze \cite{DyerFrieze}.
The difficulty in proving the lemma in our setting is that $k<<\Delta$ and thus
we have to consider frozen vertices.
Before proving Lemma~\ref{lem:unif-HD},
we now use it to
complete the proof of Theorem~\ref{thm:high-degree}.
The essential point is that colorings for which
every vertex has many available colors are universally
distance-decreasing, as defined by Hayes and Vigoda in
\cite{HV-CwS}.  Since Lemma~\ref{lem:unif-HD} implies
that almost all colorings satisfy this property,
rapid mixing follows by ``coupling with stationarity.''
Note that, following Hayes~\cite{Hayes}, we use
a weighted Hamming metric, with weights taken from the
principal eigenvector of $G$.

Assuming $Y_0$ is chosen uniformly from $\OM$, $Y_t$ is uniform over $\OM$. Therefore,
 conditioning on an event of probability $1-O(n^{-4})$ we have
\begin{align*}
\Exp{w(D_{t+1})|X_t,Y_t} - w(D_t)
&\le
\frac 1{n\D^{1-\e/2}}\sum_{v\in V}w(v)|N(v)\cap D_t|  - \frac 1nw(D_t)\\
&=  \frac 1{n\D^{1-\e/2}}\sum_{u\in D_t}\sum_{v \in N(u)}w(v) - \frac 1nw(D_t)\\
&\le \frac {\widetilde{\r}}{n\D^{1-\e/2}}\sum_{u\in D_t}w(u) - \frac 1nw(D_t)\\
&\le  -\frac {1}{2n} w(D_t),
\end{align*}
Therefore, using \eqref{eq.unif-HD}
\[
\ExpCond{w(D_{t+1})}{X_t,Y_t} \leq \brac{1-\frac {1}{2n}}w(D_{t}) + \frac {1}{n^4}w(V).
\]

By induction, for $T \geq 2n\ln(10 n )$,
we have for $n > 10$
\[
\ExpCond{w(D_{T})}{X_0,Y_0} \leq \brac{1-\frac {1}{2n}}^T \; w(D_{0}) + \frac {2}{n^3}w(V) \leq \brac{\frac 1{10n} + \frac 2{n^3}} \|w\|_1
\leq w_{\min}/4,
\]
where for the last inequality we have used \eqref{eq:wminbound}.
This proves \eqref{eq:contract-high-deg} and
completes the proof of Theorem~\ref{thm:high-degree}. \hfill\qed

Finally, we prove the uniformity result, Lemma \ref{lem:unif-HD}.
In order to deal with the possibility of  frozen vertices,
we divide the vertices into level sets
based on the principal eigenvector.
A simplified example which illustrates the intuition of the proof
is the case of the complete $(\Delta-1)$-ary tree.
To prove the uniformity property we would
first consider the leaves which are clearly not frozen.
After all of the leaves are recolored,
we can consider the parents of the leaves since these vertices are now likely
to have some colors available when $k=\Omega(\Delta/\log{\Delta})$,
and then we continue up the tree by level.

\newcommand{\Avnew}[2]{{\widetilde{A}(#1,#2)}}
\newcommand{\Av}[2]{{A(#1,#2)}}

\begin{proof}[Proof of Lemma \ref{lem:unif-HD}]
For $v\in V$ and $Y \in \OM$, define $\cG(Y,v)$ as the event that $v$ has the
desired uniformity property under $Y$, that is,
\[
 |\Aval_Y(v)| \geq \frac12 ke^{-\D/k}.
\]
Similarly, for $U \subseteq V$, let $\cG(Y,U)$ denote
the intersection of the events $\cG(Y,v)$, for all $v \in U$.
We will prove, by induction over levels,
that if $Y$ is chosen uniformly in $\OM$,
\newcommand{\ppp}{p}
\begin{equation}
\label{eq:induction111}
\Prb{\neg \cG(Y,L_{\le i})} \leq 2^i|L_{\le i}| \ppp , \text{ for all }i,
\end{equation}
where $\ppp = n^{-6}$.
It will follow that
\[
\Prb{\neg \cG(Y,V)} \leq 2^{h} n \ppp \le n^2 \ppp \le n^{-4},
\]
where the bound  $2^h \le n$ follows from \eqref{eq.bd-m}
assuming $\Delta \ge \exp(4/\eps)$.

The base case $i=0$ of \eqref{eq:induction111} follows vacuously.
Now fix $i \geq 0$ and $v \in L_{i+1}$.
All but a few neighbors of $v$ are in previous levels,
and all but a few have small co-degree with $v$.
Let $S$ be the set of vertices satisfying both properties.
Namely, let
\[
S = \{u\in N(v):u\in L_{\le i} \text{ and } |N(u) \cap N(v)| \leq \widetilde{\r}\D^{\e/2}\}.
\]
Let $\overline S = N(v)\setminus S$.
Notice that,
\[
\overline S \cap L_{\le i} \subset \{u \in N(v) : |N(u) \cap N(v)| > \widetilde{\r}\D^{\e/2}\}.
\]
Hence,
\[
|\overline S \cap L_{\le i}| \widetilde{\r}\D^{\e/2} \leq
\sum_{u \in N(v)}|N(u)\cap N(v)| \leq \r\D \leq \widetilde{\r}\D.
\]
And by simplifying, we have
\[
|\overline S \cap L_{\le i}| \leq \D^{1-\e/2}.
\]
On the other hand, by \eqref{eq:expansion-up},
\[
|\overline S \setminus L_{\le i}| \leq |N(v) \setminus L_{\le i}|
\leq \D^{1-\e/2}.
\]
Therefore,
\[
|\overline S| \leq 2\D^{1-\e/2}.
\]
Thus, all but few of the neighbors of $v$ are in $S$.  

We will recolor the vertices in $S$.
Building on the approach used in \cite{FriezeVera}, we will
use the small co-degree to show that the colors assigned to
$S$ are ``fairly independent,'' and hence that enough colors
remain available for $v$.

Let $q = |S|$ and write $S = \{s_1,s_2,\dots,s_q\}$.
We run the following experiment: Choose $Y \in \OM$ uniformly at random.
Define $Y_0 =Y$ and for
each $j = 1,\dots,q$, let $Y_{j} \in \OM$ be obtained by recoloring $s_j$ with
a color chosen uniformly from $\Aval_{Y_{j-1}}(s_j)$.
We will prove
\begin{equation}\label{eq.mainUnif}
\Prb{\neg\cG(Y_q,v)|\cG(Y,L_{\le i})} \leq \ppp.
\end{equation}
Notice that since $Y_0 = Y$ is uniformly distributed over $\OM$, so are
$Y_1, \dots, Y_q$.  This allows us to deduce
\begin{align*}
\Prb{\neg \cG(Y,L_{\leq i+1})}
&\leq \Prb{\neg \cG(Y,L_{\le i})} + \Prb{\neg \cG(Y,L_{i+1})}\\
&= \Prb{\neg \cG(Y,L_{\le i})} + \Prb{\neg \cG(Y_q,L_{i+1})}\\
&\leq 2\Prb{\neg \cG(Y,L_{\le i})} + \sum_{v\in L_{i+1}}\Prb{\neg \cG(Y_q,v)|\cG(Y,L_{\le i})}\\
&\leq |L_{\le i}|2^{i+1} \ppp + |L_{i+1}| \ppp \;\;\;\;
\mbox{by induction and \eqref{eq.mainUnif}}\\
&\leq |L_{\le i+1}|2^{i+1} \ppp
\end{align*}
To prove \refer{eq.mainUnif} we first consider the
case in which there are actually no edges between
vertices in $S$.  In this case, conditioned on $Y$,
the colors assigned to $S$ under $Y_q$ are fully
independent random variables.
\newcommand{\Amin}{a_{\min}}
Let \[
\Amin := \frac12 ke^{-\D/k} - \Delta^{1-\eps/2}.
\]
Using \eqref{eq:expansion-up}, in the case of
the good event $\cG(Y,L_{\le i})$, for each $1\le j\le q$,
the color of $s_j$ in $Y_j$
is chosen uniformly from at least $\Amin$ possibilities.

Let $K = [k] \setminus Y(\overline{S})$ denotes
the set of colors which could
possibly be available to $v$ under $Y_q$, given $Y$.
Following Dyer and Frieze~\cite{DyerFrieze}, we have the following
chain of inequalities:
\begin{align}
\nonumber
\ExpCond{|\Aval_{Y_q}(v)|}{Y}
&= 
\sum_{c \in K} \prod_{j=1}^{q}
\left( 1 - \frac{\ind{c \in \Aval_Y(s_j)}}{|\Aval_Y(s_j)|}\right)  
\nonumber
\\
\nonumber
&\ge 
|K|
\prod_{c \in K} \prod_{j=1}^{q}
\left(1 - \frac{\ind{c \in \Aval_Y(s_j)}}{|\Aval_Y(s_j)|}\right)^{1/|K|} \\
& \qquad \qquad \qquad \qquad
\mbox{(by the arithmetic-geometric mean inequality)}
\nonumber
\\
\nonumber
&\ge 
|K|
\exp \left(- \frac{1}{|K|} \sum_{c \in K} \sum_{j=1}^q
\frac{\ind{c \in \Aval_Y(s_j)}}{|\Aval_Y(s_j)| - 1}\right) \\
\nonumber
&\ge 
|K|
\exp \left(- \frac{1}{|K|} 
\sum_{j=1}^q
\frac{|\Aval_Y(s_j)|}{|\Aval_Y(s_j)| - 1}\right) \\
\nonumber
&\ge \ind{\cG(Y,L_{\le i})}
|K| \eee^{-q \Amin/|K| (\Amin - 1)} \\
&\ge \ind{\cG(Y,L_{\le i})} \frac{9}{10}k \eee^{-\Delta/k}.
\label{exp-HD}
\end{align}

Now consider the (Doob) martingale $Z_0, \dots, Z_q$
defined by
\[Z_j = \ExpCond{|\Aval_{Y_q}(v)|}{Y, Y_1(s_1),\dots,Y_j(s_j)}.
\]
Note that $Z_0 = \ExpCond{|\Aval_{Y_q}(v)|}{Y}$,
while $Z_q = |\Aval_{Y_q}(v)|$.
Because the colors $Y_j(s_j)$ are independent, conditioned on $Y$,
and each step reveals only a single color,
it follows that  $|Z_j - Z_{j-1}| \le 1$.  Hence the
Azuma-Hoeffding
 inequality (c.f., \cite{ProbMeth})
 yields
\begin{align}
\nonumber
\PrbBigCond{|\Aval_{Y_q}(v)| \le \frac{8}{10}k \eee^{-\Delta/k}}{ \cG(Y,L_{\le i}) }
&\le \PrbBigCond{Z_q \le Z_0 - \frac{1}{10}k \eee^{-\Delta/k}}{ \cG(Y,L_{\le i}) } 
\\
\nonumber
&\le \exp\left(- \left(\frac{1}{10}k \eee^{-\Delta/k}\right)^2/2\Delta \right)
\\
&\le p/2
\label{conc-HD}
\end{align}
where in the last step we used the relations
\[
k \eee^{-\Delta/k} \ge k \Delta^{-\eps/4} \ge \Delta^{1 - \eps/3}
\]
\[
\Delta \ge (\ln n)^{1 + \eps}
\]
and that $\Delta$ is sufficiently large as a function of $\eps$.
This completes the proof of \eqref{eq.mainUnif} in the
case when there are no edges within $S$.

For the general case, we argue that,
assuming $\cG(Y,L_{\le i})$, the edges within $S$
cause a negligible effect on $Y_q$.  To
this end, couple the
recolorings $Y_0 = Y, Y_1, \dots, Y_q$ on the actual
graph with the corresponding recolorings
$\widetilde{Y}_0 = Y, \widetilde{Y}_1, \dots, \widetilde{Y}_q$
on the graph with the edges within $S$ deleted.
Define the coupling by induction,
at each step maximizing the
probability that $\widetilde{Y}_j(s_j) = Y_j(s_j)$,
conditioned on the history.

Now, by the definition of $S$ and because we are assuming the
good event $\cG(Y,L_{\le i})$,
for each $1\le j\le q$, the coupled recoloring of $s_j$ in 
$(Y_{j-1},\widetilde{Y}_{j-1})\rightarrow
(Y_{j},\widetilde{Y}_{j})$
has at most a
$ \Delta^{1- \eps/2}/\Amin \le \Delta^{-\eps/2}$
probability to create a disagreement (in the sense that 
$Y_j(s_j)\neq \widetilde{Y}_j(s_j)$).

Now by comparison with a sequence of independent
coin flips, we see that the probability of
having at least
$\frac1{10}k\eee^{-\Delta/k}$ disagreements is at most
\begin{align*}
\binom{\Delta}{\frac{1}{10}k\eee^{-\Delta/k}}
\left(\Delta^{-\eps/2}\right) ^{k\eee^{-\Delta/k}/10} &\le
\left(\frac{\eee \Delta}{\frac{1}{10}k \eee^{-\Delta/k} \Delta^{\eps/2}}
\right)^{k \eee^{-\Delta/k}/10} \\
&\le \left(\frac{10 \eee}{\Delta^{\eps/6}}\right)^{\Delta^{1 - \eps/3}/10} \\
&\le p/2.
\end{align*}
This combined with \eqref{conc-HD}
proves \eqref{eq.mainUnif} in the general case,
completing the proof of Lemma~\ref{lem:unif-HD}.
\end{proof}

\section{Graphs of low degree: Setup}
\label{sec:constant-degree}

In this section we restrict attention to planar graphs, and
all of the statements are for planar graphs with
maximum degree $\Delta$ with $k>100\Delta/\log{\Delta}$
colors and $\Delta$ sufficiently large.
Fix $\eta=1/30$.  We use that 
$\eta$ is a small
constant, but we also will use that $k\exp(-\Delta/k) >> \Delta^{1-\eta}$
(i.e., $\eta>> 1/100$).

In Section \ref{sec:properties} we pointed out several
important properties of the level sets for planar graphs.
Recall, for planar graphs we have that
$\rho=\Delta^{1-\eps}/2\leq6\Delta^{1/2}$ (see Section \ref{sec:spec-radius})
and hence
$\eps\geq 1/2-\eta$.
Therefore, by \eqref{eq:neighbor-weight}, for every vertex $v$,
\begin{equation}
\label{eq:planar-neighbor-weight}
w(N(v)) \le \Delta^{1/2+\eta} w(v).
\end{equation}

We denote the down-neighbors of $v$ as
\[
 N^-(v) = \{u\in N(v): \ell(u) < \ell(v)\},
\]
 and let $N^+(v) = N(v) \setminus N^-(v)$ denote the up-neighbors.
 For planar graphs, by \eqref{eq:expansion-up} we know that the
 number of up-neighbors for any vertex $v$ satisfies:
 \[ 
 |N^+(v)| \leq \Delta^{1-\eps/2} < \Delta^{5/6}.
 \]

Recall, for a vertex $v$, $\ell(v)$ denotes
the level of $v$, and $M = \max_{\{u,v\}\in E(G) }|\ell(v) - \ell(u)|$.
In Section \ref{sec:properties},
by \eqref{eq:bound-M} we observed that for planar graphs
we have $M\le 2$.

We first describe the main challenge when $\Delta$ is small
and try to provide some intuition about how we overcome this
obstacle.
When the maximum degree is constant, in a random coloring,
a constant fraction of the vertices might be frozen.
This poses a problem as the set of disagreeing vertices
under our coupling may be highly correlated with the frozen vertices
in the two colorings.
To see the difficulty, consider the complete $(\D-1)$-ary tree
with a single disagreement at the root $v$. Suppose all vertices
except the leaves
have very few available colors
 (we will later refer to these vertices with few
 available colors as nearly frozen).
 Then in the early stage of the dynamics neighbors of disagreements have few colors available, and thus might have a high probability of becoming a disagreement.

The proof of Lemma \ref{lem:unif-HD} gives some insight
on how to overcome the difficulty of frozen vertices to try to get some sort of
independence between the probability that different vertices are frozen.
In that proof, a tree-like structure of the graph is exploited recoloring
the graph from the ``leaves'' up. In that way the uniformity property
propagates through the tree structure of the graph.
By using our level-set dynamics
where the sets correspond to the level sets based on the principal eigenvector
we can achieve similar behavior.  Once again, vertices will have the uniformity
property with probability roughly $1-\exp(-\Delta^{1/2})$, but in this case that means
a constant fraction of the vertices will not have the  uniformity property.  The key
is that the graph within a level set is sparse and most neighbors of this set
are in earlier sets.  Consequently, we will get that vertices within a set
are roughly independent of each other, in terms of having the uniformity property.

\subsection{Coupling Proof Setup}
\label{sec:setup}

Consider an arbitrary pair of colorings $X_0$ and $Y_0$, we will analyze Jerrum's
coupling for this pair (see Section \ref{sec:coupling}).  We analyze one pass of the level-set dynamics over the whole graph.
Recall that, starting with $i=0$, the level-set dynamics performs $T_i = |L_i|\ln \D$ random Glauber steps in level $L_i$ and then moves to the next level $L_{i+1}$. 
Let  $X_{i,t}$ and $Y_{i,t}$ denote the colorings after $t$ steps in level $i$.
Hence, $X_{0,0} = X_0$ and $X_{i+1,0}=X_{i,T_i}$. Recall, $h$ is the total number of levels in $G$. Thus $X_{h,T_h}$ is the coloring obtained after one pass of the level-set dynamics.

For $i>0$ and $t\geq 0$, let $D_{i,t}$ denote the set of disagreements at time $t$ in round $i$, i.e.,
\[
D_{i,t} = \{v\in V: X_{i,t}(v)\neq Y_{i,t}(v)\}.
\]

Our main result will be that the weight of disagreements
decrease after one scan of the graph by the level-set dynamics.
\begin{lemma}\label{lem:contractionLow}
For any colorings $X_{0,0}$ and $Y_{0,0}$,
\begin{equation*}
\ExpCond{w(D_{h,T_h})}{X_{0,0},Y_{0,0} } \leq \Delta^{-1/3} w(D_{0,0})
\end{equation*}
\end{lemma}

We will prove Lemma \ref{lem:contractionLow} in section \ref{subsec:WeightContraction}.
Assuming Lemma \ref{lem:contractionLow} we can prove rapid
mixing of the level-set dynamics, thereby establishing
Theorem~\ref{thm:constant-degree}(i).

\begin{proof}[Proof of Theorem \ref{thm:constant-degree}(i)]

Recall, a scan of the level-set dynamics recolors levels $0,\dots,h$.
Let $X^j_{i,t}$ and $Y^j_{i,t}$ denote the dynamics in the $j$-th scan, during
step $t$ in round $i$ (i.e., within level $L_i$).
Similarly, define $D^j_{i,t}$ as the set of disagreements between
  $X^j_{i,t}$ and $Y^j_{i,t}$.
 Let
 \[
N:=3\lceil \ln( 4\|w\|_1/ w_{\min})/\ln(\Delta)\rceil.
 \]
 Note that $N=O(\log{n})$ by \eqref{eq:wminbound}.

 For any initial pair of colorings $X_{0,0}^0$ and $Y_{0,0}^0$, we have that:
   \begin{align*}
    \PrbCond{X_{h,T_h}^N  \neq Y_{h,T_h}^N}{ X_{0,0}^0,  Y_{0,0}^0}
    &=
\PrbCond{D^N_{h,T_h} \geq 1}{ X_{0,0}^0,  Y_{0,0}^0}
\\
&\leq
  \ExpBCond{   \left| D^N_{h,T_h}\right| }{ X_{0,0}^0,  Y_{0,0}^0}
  \\
  &\leq
  \ExpBCond{   \frac{  w(D^N_{h,T_h}) }{  w_{\min}  } }{ X_{0,0}^0,  Y_{0,0}^0}
\\
&\leq
\Delta^{-N/3} w(D_{0,0})/w_{\min} & \mbox{ by Lemma \ref{lem:contractionLow}}
\\
&\leq   \frac { w(D_{0,0}) }{4\|w\|_1}
& \mbox{ by the definition of $N$}
\\
&\leq  1/4.
\end{align*}
By the coupling inequality (see Section \ref{sec:coupling}), this proves that after $N=O(\log{n})$ scans
of the graph, the level-set dynamics is within variation distance $\le 1/4$ of the stationary distribution.
\end{proof}

\subsection{Contraction of the Weight of Disagreements}\label{subsec:WeightContraction}

In this subsection we prove Lemma \ref{lem:contractionLow}.

For a pair of vertices $u \in V$ and $w \in L_{i}$ for $i\in\{0,\dots,h\}$,
 we say that a {\em crossing}
from $u$ to $w$ occurs at time $(i,t)$ in the coupling between $X_{i,t}$ and $Y_{i,t}$, 
and denote it by $\cC_{i,t}(u,w)$, if $w$ is the updated vertex at time $t$ in round $i$
and $Y_{i,t+1}(w) = X_{i,t}(u)$ or $X_{i,t+1}(w) = Y_{i,t}(u)$ (i.e., $w$ 
is colored in one chain with $u$'s color in the other chain).

Notice that if $w\in L_{i}$ becomes a disagreement at time $(i,t)$
(i.e., $X_{i,t+1}(w)\neq Y_{i,t+1}(w)$ and $X_{i,t}(w)=Y_{i,t}(w)$) 
then $v$ has a neighbor $u \in D_{i,t}$ 
such that $\cC_{i,t}(u,w)$ occurs.
Hence, if $w \in D_{h,T_h}$, then the disagreement at $w$ 
can be traced back to a disagreement in $D_{0,0}$ through a path of 
crossings (note,
the definition of crossings does not depend on disagreements).

For a path $\sigma = (w_0,w_1,\dots,w_s)$, let $\cC(\sigma)$ be the event
\[
\exists t_1,\dots,t_s\, \bigwedge_{i=1}^s \cC_{\ell(w_i),t_i}(w_{i-1},w_i),
\]
that there is a sequence of crossings from $w_0\rightarrow w_1 \rightarrow \dots \rightarrow w_s$.

Notice that disagreements only propagate through paths 
where $\ell_1 \leq \ell_2 \leq \dots \leq \ell_s$.  Thus,
these paths are ``up-paths'', except that $\ell_0$ might be larger than $\ell_1$.
 Moreover, we can assume these paths are loopless.
Let $\cP_s(v,z)$ denote the set of such loopless up-paths 
of length $s$ from $v$ to $z$, and let $\cP(v,z) = \bigcup_{s\geq 1} \cP_s(v,z)$.

We will show that for any path $\sigma\in\cP(v,z)$, independently of the initial colorings, the probability that $\cC(\sigma)$ occurs 
decays exponentially in the  length of $\sigma$.  Namely, we will show that, for 
all $X_{0,0},Y_{0,0}$, all vertices $v,z$, all $\sigma\in\cP(v,z)$,
\begin{align}\label{eq:explen}
\PrbCond{\cC(\sigma)}{X_{0,0},Y_{0,0}} \leq \Delta^{-9|\sigma|/10}.
\end{align}
Assuming $\eqref{eq:explen}$ we can prove Lemma \ref{lem:contractionLow} as
follows.
\begin{proof}[Proof of Lemma \ref{lem:contractionLow}]
For every $z \in D_{h,T_h}$, either (i) this disagreement was there initially 
and $z$ was never recolored, or (ii) the disagreement at $z$ 
can be traced back to a $v \in D_{0,0}$ via a path of crossings as discussed earlier.
For case (i) to occur, it must be that $X_{0,0}(z)\neq Y_{0,0}(z)$ and with probability
$\left(1- \frac 1{|L_{\ell(z)}|}\right)^{ |L_{\ell(z)}|\ln{\Delta}}$
vertex $z$ was never recolored.  For case (ii) to occur, 
there is a path $\sigma \in \cP(v,z)$ such that $\cC(\sigma)$ holds.
Hence,
\begin{align*}
&\ExpCond{w(D_{h,T_h})  }{X_{0,0},Y_{0,0} }\\
& =
\sum_{z\in V} w(z)\PrbCond{z\in D_{h,T_h}}{X_{0,0},Y_{0,0}}
\\
&\leq
\sum_{v\in D_{0,0}} w(v)
\PrbCond{v \text{ is never recolored }}{X_{0,0},Y_{0,0}}
+ \sum_{z\in V} w(z) \sum_{v\in D_{0,0} }
\sum_{s\geq 1} \sum_{\sigma\in\cP_s(v,z)}
\PrbCond{\cC(\sigma)}{X_{0,0},Y_{0,0}}
\\
&\leq
\sum_{v\in D_{0,0}}w(v)\brac{ 1- \frac 1{|L_{\ell(z)}|}}^{ |L_{\ell(z)}|\ln{\Delta}}
+ \sum_{v\in D_{0,0} } \sum_{s\geq 1}
 \sum_{z\in V: \atop \cP_s(v,z)\neq\emptyset} w(z)
\sum_{\sigma\in\cP_s(v,z)}
\Delta^{-9s/10} 
\qquad \qquad \qquad \mbox{ by \eqref{eq:explen}}
\\
&\leq
 \sum_{v\in D_{0,0} } w(v)
\brac{\frac 1{\Delta} + \sum_{s\geq 1}
\Delta^{-(4/10-\eta)s}} \qquad \qquad \qquad \mbox{ by \eqref{eq:planar-neighbor-weight}}
\\
&\leq  w(D_{0,0})\Delta^{-1/3}.
\end{align*}
\end{proof}

Hence, to complete the proof of Part (i) of Theorem \ref{thm:constant-degree}
it remains to prove \eqref{eq:explen}.

\section{Exponential Decay of Crossing Probabilities}
\label{sec:tech-proofs}

Here we prove \eqref{eq:explen}, which says 
that $\PrbCond{\cC(\sigma)}{X_{0,0},Y_{0,0}}$ decays exponentially on  the length of $\sigma$. Given that $w$ is chosen to be recolored at time $t$ in round $i$, 
the probability of a crossing from $u$ to $w$, depends on the number of 
available colors for $w$. If $w$ has more than $\Delta^{1-\eta}$ available colors, then the probability of a crossing is $\leq 2\Delta^{-1+\eta}$. We will show that the probability of $w$ being nearly frozen (i.e. having less than $\Delta^{1-\eta}$ available colors) is $\exp(-\Delta^{1/2})$,
 and thus the probability for a crossing  from $u$ to $w$ is $\leq 3\Delta^{-1+\eta}$. The problem with proving \eqref{eq:explen} is then showing that crossings through the edges of $\sigma$ occur in a relatively independent way. In particular, we need to show that  vertices in $\sigma$ should become nearly frozen in a manner which is (almost) indpendent of
 the crossings.
  We do this by analyzing a more general dynamics
  (which we call an adaptive adversarial dynamics) 
  where almost independence is evident for a carefully chosen subset of $\sigma$ 
  of size $99|\sigma|/100$.

\subsection{Adversarial Dynamics}
\label{sec:adversary-dynamics}

We will prove \eqref{eq:explen} for the coupling of two copies of a generalized coloring process which
we call {\em adaptive adversarial level-set dynamics}.
For our input graph $G=(V,E)$ with maximum degree $\Delta$, we consider
a subgraph $H=(V,E_H)$ where $E_H\subset E$.

In the adaptive adversarial level-set dynamics run on $H$,
we use the level sets $L_0, L_1,\dots,L_{h}$ defined based on $G$.
The adaptive dynamics, denoted by $\wX_{i,t}$, works the same as the level-set dynamics run on $G$,
except that when we choose a vertex $v$ to update at time $t$ in round $i$,
then in the adaptive dynamics, the adversary can choose any at most
$|N_G(v)\setminus N_H(v)|$ additional colors to block for $v$.  
The adversary can look
at $\widehat{X}_{i,t-1}$ to decide on these colors.  Thus, in the adaptive dynamics on $H$,
it is as if the edges we deleted to form $H$
are replaced by edges to new vertices, and the adversary controls the colors
of these vertices and can change their colors at their will.

We will couple two adversarial dynamics, $\wX$ and $\wY$, run on $H$. We assume that the two adversaries can collaborate, equivalently, 
there is only one adversary making decisions for both processes. 
Once again the coupling we use is
Jerrum's coupling as defined in Section \ref{sec:coupling}.
Hence, given $\wX_{i,t}$ and $\wY_{i,t}$, 
we choose a random vertex $v$ for update (in both chains), 
then the adversary can use $\wX_{i,t}$ and $\wY_{i,t}$ to choose 
the $\leq |N_G(v)\setminus N_H(v)|$ additional 
colors to block for $v$ in $\wX$ and in~$\wY$.

Finally, the updated color for $v$ in $\wX_{i,t+1}$ and $\wY_{i,t+1}$ 
is coupled to maximize the probability of choosing the same color. 
Notice that (if $|N_G(v)\setminus N_H(v)| > 0$) 
the adversary can always choose the colors in such a way that there is some positive probability of creating a disagreement. So we can not expect disagreements to disappear in this generalized process, instead,
assuming an upper bound on how many colors the adversary can 
block for any vertex, we can still prove a related form of \eqref{eq:explen}.

We can imagine the goal of the adversary being to increase the probability of $\cC(\sigma)$. By definition, a crossing from $u$ to $v$ is the event of 
$v$ receiving $u$'s color either in $X$ or in $Y$. To generalize the notion of crossing to the adversarial setting, we would allow the adversary to select
a pair of colors $c_X,c_Y$.
If $v$ receives color $c_Y$ in $X$ or $c_X$ in $Y$
then we say an adversarial crossing for $v$ has occured.
The choice of the forbidden colors is a dynamical choice for the adversary; 
so at each time $t$ in round $i$, after choosing the vertex $v$ for update,
the adversary picks the two colors $c_X$ and $c_Y$. 
Hence, for the adversarial crossing to occur for $v$ at
this time $(i,t)$ we see if 
$X_{i,t+1}(v) = c_Y$ or $Y_{i,t+1}(v) = c_X$.
Given a set of vertices $U= \{u_1,\dots,u_s\}$, we say that the adversary has crossed $U$, denoted $\cK(U)$, if the adversary has crossed $u_1,\dots,u_s$
for some set of times (in any order).

Before stating the main lemma of this section,
we formally define the set of  nearly-frozen vertices in $\widehat{X}_{i,t}$ as:
\[
  \widehat{\cF}_{i,t} = \left\{v \in V: \left|\Aval_{\widehat{X}_{i,t}}(v)\right| \leq 2\Delta^{1-\eta}\right\}.
\]

\begin{lemma}\label{lem:mainAdv}
There exist $\Delta_0>0$ such that  for all $m\ge 0$, all planar graphs $G=(V,E)$ of maximum degree
$\Delta>\Delta_0$, all $S \subseteq  L_{\leq m}$ and all $H=(V,E_H)$ with $E_H\subset E$   if:
 \begin{enumerate}[(a)]
\item
for all $v\in V$,
$|N_G(v)\setminus N_H(v)| + |N^+_{H}(v)|\leq \Delta^{1-2\eta}$, and
\item $\sum_{v\in V} |N_G(v)\setminus N_H(v)| \leq\Delta^{1- 4\eta}|S|$,
\end{enumerate}
then,
\begin{enumerate}
\item \label{part:uniformity}
for the adaptive adversarial level-set dynamics $(\widehat{X}_{i,t})$ on $H$, for any $\widehat{X}_{0,0}$ and any adversary,
\begin{equation*}
\PrbCond{S \subset \widehat{\cF}_{m,0}}{\widehat{X}_{0,0}} \leq p^{|S|},
\end{equation*}
where $p = e^{-\Delta^{1/2}}$.
\item \label{part:explen}
and in the coupling of the adversarial dynamics,
for any initial colorings $\wX_{0,0}$ and $\wY_{0,0}$ of $H$ and any adversary,
\begin{equation*}
\PrbCond{\cK(S)}{\wX_{0,0},\wY_{0,0}} \leq \Delta^{-9|S|/10}.
\end{equation*}
\end{enumerate}
\end{lemma}

Part \ref{part:uniformity} of Lemma \ref{lem:mainAdv} refers to the so called uniformity properties. Usually this type of Lemma is only proved in the case for a single vertex (i.e. the case $|S| =1$), here we prove a stronger form. 
Part \ref{part:explen} is a stronger form of \eqref{eq:explen}.  
To obtain \eqref{eq:explen} from it, take $H = G$, $\wX_{0,0} = X_{0,0}$, $\wY_{0,0} = Y_{0,0}$, $S = \{\sigma_i:i=1,\dots,|\sigma|\}$ and take as an adversary the one that to cross $\sigma_i \in L_{\ell_i}$, once $\sigma_i$ is chosen, 
sets $c_X= \wX_{\ell_i,t-1}(\sigma_{i-1})$ and 
$c_Y = \wY_{\ell_i,t-1}(\sigma_{i-1})$.
 Notice that for this adversary, $\cC(\sigma)$ implies $\cK(\sigma)$.

We prove Lemma \ref{lem:mainAdv} by induction on $m$.  We will construct $S^* \subseteq S$ and $H^* \subseteq H$ such that $|S^*| \geq 99|S|/100$ and in $S^*$ we have enough independence in the adversary process run on $H^*$. Then we will apply our Lemma inductively on $N^-(S^*) \subseteq L_{\leq m-1}$ in $H^*$.

\subsection{Structural Lemma}

Our first lemma captures the important structural properties of $S^*$ and $H^*$, the proof of which uses planarity.  
For a path $v_1,v_2,\dots,v_j$ where $j> 1$, we call this an {\em up-path} if
for all $1\le i<j$, $v_{i+1}\in N^+(v_i)$, i.e., the levels are non-decreasing.

\begin{lemma}\label{lemma:structural}
Let $G=(V,E)$ be a planar graph with maximum degree $\Delta$. Let $H=(V,E_H)$ be a subgraph of $G$ and $S\subset V$  such that the following hold:
\begin{enumerate}[(a)]
\item
\label{hyp:small-adv}
For all $v\in V$,
$|N_G(v)\setminus N_H(v)| + |N^+_{H}(v)|\leq \Delta^{1-2\eta}$.
\item $\sum_{v\in V}  |N_G(v)\setminus N_H(v)| \leq\Delta^{1- 4\eta}|S|$.
\end{enumerate}
Then there exists $H^*=(V,E^*)$
 with $E^*\subset E_H$, and  $S^*\subset S$
with the following properties:
\renewcommand{\theenumi}{{\bf P-\arabic{enumi}}}
\begin{enumerate}
\item $|S^*| \geq 99|S|/100$.
\label{prop:size}
\item
\label{prop:changes}
\begin{enumerate}
\item
For all $v\in V$,
$|N_G(v)\setminus N_{H^*}(v)| + |N^+_{H^*}(v)|\leq \Delta^{1-2\eta}$.
\item $\sum_{v\in V}  |N_G(v)\setminus N_{H^*}(v)| \leq\Delta^{2-8\eta}|S^*|$.
\end{enumerate}
\item For all $v\in V$,
\label{prop:lipschitz}
\[
\left| N^+_{H^*}(v) \cap S^* \right| \leq 30.
\]
\item
For all $v,w\in S^*$, for all $y\in N^-_{H^*}(v)$, all $z\in N^-_{H^*}(w)$, if $y \neq z$, then there is 
\label{prop:independence}
\[  \mbox{ no up-path  from $y$ to $z$.}\]
\end{enumerate}
\end{lemma}

\subsection{Proof of Lemma \ref{lem:mainAdv}: Uniformity Properties of the Adversary Dynamics}

Throughout the proof, for various inequalities we will use that
$\Delta$ is sufficiently large. We prove the lemma by induction on $m$.  Hence, we fix $m$, and we assume Lemma \ref{lem:mainAdv} holds for all $m'<m$.

We first apply Lemma \ref{lemma:structural} to $G$, $H$ and $S$ from the hypothesis of Lemma \ref{lem:mainAdv},
obtaining $S^* \subseteq S$ and $H^* \subseteq H$.
Using Property \ref{prop:changes}, we can apply our induction hypothesis for $H^*$, and any $S' \subseteq L_{\leq m-1}$ such that
$|S'| \geq \Delta^{1-4\eta}|S^*|$.

As $S^* \subseteq S \subseteq L_{\leq m}$, we have $N^-_{G^*}(S^*) \subseteq L_{\leq m-1}$. Applying our induction hypothesis to all $S' \subseteq N^-_{G^*}(S^*)$ of size  $\Delta^{1-4\eta}|S^*|$ we obtain:
\begin{align}
\nonumber 
\lefteqn{ \hspace*{-1in}
\PrbBigCond{|N^-_{G^*}(S^*) \cap \widehat{\cF}_{X_{m-1,0}}| \geq \Delta^{1-4\eta}|S^*|}{\wX_{0,0}}
} \\
\nonumber
& \leq \PrbBigCond{\exists\, S' \subseteq N^-_{G^*}(S^*): |S'| = \Delta^{1-4\eta}|S^*|, S' \subset \widehat{\cF}_{X_{m-1,0}}}{\wX_{0,0}}\\
\nonumber
&\leq {|N^-_{G^*}(S^*)| \choose \Delta^{1-4\eta}|S^*|} p^{\Delta^{1-4\eta}|S^*|}\\
\nonumber
&\leq {\Delta |S^*| \choose \Delta^{1-4\eta}|S^*|} p^{\Delta^{1-4\eta}|S^*|}\\
\nonumber
& \leq (ep\Delta^{4\eta})^{ \Delta^{1-4\eta}|S^*|}
\\
& \leq p^{4|S^*|}
\label{eq:induction}
\end{align}

Thus we can assume that most vertices in $N^-_{G^*}(S^*)$ have the uniformity
property, this will be one of the keys to the proof.
Note, the statement of Lemma \ref{lem:mainAdv} is about the uniformity property for $S$ which means
that it is a property of the colors assigned to $N^-(S)$.
By Property { \ref{prop:size}}, it suffices to prove the uniformity property for
$S^*\subset S$.  We will use the Dyer-Frieze approach \cite{DyerFrieze},
similar to what we did in the derivation of \eqref{exp-HD} in Section \ref{sec:high},
to each $v \in S^*$ to obtain the desired uniformity property
for most vertices in $S^*$. 
Property {\ref{prop:independence}} will grant us enough independence among vertices in $S^*$.

Let $\widehat{X}^*_{i,t}$ denote the adaptive adversarial level-set
dynamics on $H^*$ (a subgraph of $G$), as defined in Section
\ref{sec:adversary-dynamics}.
We also define a new chain
$\tX^*$ on $H^*$.  Let $R$ be the set $N^-_{H^*}(S^*)$ and their ancestors.
The chain $\tX^*$ is the same as $\widehat{X}^*$ except that
we defer the updates of vertices in $R$
until the end.  Let $T$ be the time when we start to recolor the deferred set of vertices.
By Property {\ref{prop:independence}}, no vertex in $N^-_{H^*}(N^-_{H^*}(S^*))$ is an ancestor of another vertex in
$N^-_{H^*}(S^*)$.  Therefore, no vertex in $N^-_{H^*}(N^-_{H^*}(S^*))$ is in $R$.
Thus all of $N^-_{H^*}(N^-_{H^*}(S^*))$ is recolored before time $T$, which implies that:
\begin{equation}
\label{W-Xstar222}
 \tX^*_{T}(N^-_{H^*}(N^-_{H^*}(S^*))) = \widehat{X}^*_{m-1,0}(N^-_{H^*}(N^-_{H^*}(S^*))).
\end{equation}

Using Property {\ref{prop:independence}} again, no vertex in $N^-_{H^*}(S^*)$
is an ancestor of another vertex in $N^-_{H^*}(S^*)$. So, after time $T$, when recoloring the deferred set $R$,  we can first recolor $N^-_{H^*}(S^*)$
and then $S^*$, before considering their ancestors.

Let $T'$ be the time when we finish recoloring $N^-_{H^*}(S^*)$ and $T''$ the time when we finish recoloring $S^*$.
Note that:
\begin{equation}
\label{W-Xstar111}
  \tX^*_{T'}(N^-_{H^*}(S^*)) = \widehat{X}^*_{m,0}(N^-_{H^*}(S^*)),
\end{equation}
and the adversary crosses $S$ in $\wX^*$ if and only if it has crossed $S$ in $\tX^*$ by $T''$.

We will not consider $\tX^*_t$ after time $T''$.

Now we're going to prove part 1 of the lemma statement for $\tX^*$ for set $S^*$ at time $T'$.
By \eqref{W-Xstar111}, this implies the lemma statement for $\widehat{X}^*_{m,0}$ for set $S^*$.
To do this we will apply the lemma inductively to conclude that most
of $N^-_{H^*}(S^*)$ has the uniformity property in $\widehat{X}^*_{m-1,0}$.
The uniformity property for $N^-_{H^*}(S^*)$ is a function of the colors of $N^-_{H^*}(N^-_{H^*}(S^*))$.
By \eqref{W-Xstar222}, this implies that most of
$N^-_{H^*}(S^*)$ has the uniformity property in $\tX^*_T$.
In $\tX^*_T$, we know that for most $v\in S^*$, few of their neighbors are frozen.
Hence, we can apply the Dyer-Frieze approach \cite{DyerFrieze} (as in \eqref{exp-HD}) to each $v\in S^*$, to argue that
with high probability, $v$ has the uniformity property in $\tX^*_{T'}$.
By the construction of $H^*$ we will be able to argue that these vertices in $S^*$
are independently getting the uniformity property in $\tX^*_{T'}$.  To do this, we
will use the Azuma-Hoeffding inequality, where Property {\ref{prop:lipschitz}}
will be used to bound the Lipschitz constant.

For $v \in S^*$, let
\[Z_v = \left|\Aval_{\tX^*_{T'}}(v)\right|  \text{ and }  Z = \sum_{v\in S^*}Z_v.
\]
Notice that
\[
S \subset \widehat{\cF}_{\tX^*_{T'}} \text{ implies }  Z \leq 2\Delta^{1-\eta}|S^*|.
\]
Therefore, to prove part 1 of Lemma \ref{lem:mainAdv}, it is enough to show that:
\begin{equation}\label{eq:Z}
 \PrbBigCond{Z \leq 2\Delta^{1-\eta} |S^*|}{\wX_{0,0}} \leq  p^{2|S^*|}.
\end{equation}

\begin{proof}[Proof of \eqref{eq:Z}]
 We will analyze the colors assigned to vertices in $N^-_{H^*}(S^*)$.
 Let $N^-_{H^*}(S^*) = \{v_1,\dots,v_d\}$.
  Let $x_1, x_2,\dots, x_d$ be the colors assigned to $v_1,v_2,\dots,v_d$ in $\tX^*_{T'}$.
  We can write $Z = Z(x_1, x_2,\dots,x_d)$.
  Let $Z_i = Z_i(x_1, x_2,\dots,x_i) = \ExpCond{Z}{\tX^*_T,x_1,x_2,\dots,x_i}$.
  Then, $Z_i$ is a martingale with $Z_0 = \ExpCond{Z}{\tX^*_T}$ and $Z_d = Z$.

We will now argue
using  Azuma-Hoeffding's inequality that $Z$ is concentrated (around its mean).
Property {\ref{prop:lipschitz} } says that each $v_i\in N^-_{H^*}(S^*)$ has at most 30 up-neighbors in
$S^*$.  Thus, the function $Z(x_1,\dots,x_d)$ is Lipschitz, namely,
\[
|Z(x_1,\dots, x_{i-1},x_i,x_{i+1},\dots,x_d) - Z(x_1,\dots, x_{i-1},x_i',x_{i+1},\dots,x_d)| \leq 30,
\]
since for each $v_i\in N^-_{H^*}(S^*)$, changing the color for $v_i$ can affect $\Aval_{\tX^*_{T'}}(\cdot)$ only for the neighbors of $v$, and by at most one color.

Applying Azuma-Hoeffding's inequality, we obtain for any $\alpha>0$,
  \begin{equation}
   \label{Pr0-M1}
 \PrbBigCond{Z \leq \ExpCond{Z}{\tX^*_T}  -\alpha }{\tX^*_T}
 \leq \exp\brac{\frac{-\alpha^2}{2\cdot 30^{2} d}}
 \leq  \exp\brac{\frac{-\alpha^2}{1800\Delta |S^*|}}.
 \end{equation}

Let $\cU$ be the event that:
\[
|N^-_{H^*}(S^*) \cap \widehat{\cF}_{\tX^*_{T}}| \leq \Delta^{1-4\eta}|S^*|.
\]
We will prove the following inequality:
\begin{equation}
\label{ineq:M1-improved-ExpZ}
 \ExpCond{Z}{\tX^*_T} \geq |S^*|\Delta^{1-\eta/2} \ind{\cU}
\end{equation}

   Assuming \eqref{ineq:M1-improved-ExpZ} we can complete the proof of \eqref{eq:Z} as follows.

Let
\begin{eqnarray*}
  \alpha &  :=  &
  \max\{0, \ExpCond{Z}{\tX^*_T} -  2\Delta^{1-\eta}|S^*|\}
\\
& \geq  &
  \max\{0,|S^*|\Delta^{1- \eta/2}\ind{\cU} -  2\Delta^{1-\eta}|S^*|\}
\\
& \geq &
|S^*|\frac{\Delta^{1-\eta/2}}{2}\ind{\cU}
\end{eqnarray*}
Using \eqref{Pr0-M1}, we have:
\begin{equation}
\label{eq:apply-azuma-M1}
  \PrbBigCond{Z \leq 2\Delta^{1-\eta} |S^*| }{\tX^*_T } \leq  \exp\brac{\frac{-|S^*|\Delta^{1-\eta} \ind{\cU }}{7200}},
\end{equation}
and thus,
\begin{eqnarray}
\nonumber
 \PrbBigCond{Z \leq 2\Delta^{1-\eta} |S^*|}{\wX_{0,0}}
 & \leq &
  \exp\brac{\frac{-\Delta^{1-\eta}|S^*| }{7200}}
\PrbCond{\cU}{\wX_{0,0}} + \PrbCond{\neg \cU}{\wX_{0,0}}
\\
&\leq &
\nonumber
\exp\brac{\frac{-\Delta^{1-\eta}|S^*| }{7200}}
 + \PrbCond{\neg \cU}{\wX_{0,0}}\\
 &\leq&
p^{4|S^*|}
 + \PrbCond{\neg \cU}{\wX_{0,0}}\quad \quad \text{Using $p= e^{-\Delta^{1/2}}$}
 \label{ineq:M1-finalExpZ}
\end{eqnarray}

From \eqref{eq:induction}
\begin{align*}
\PrbCond{\neg \cU}{\wX_{0,0}} \leq  p^{4|S^*|}
\end{align*}

Plugging this into \eqref{ineq:M1-finalExpZ} we have:
\begin{align*}
 \Prb{Z \leq \Delta^{1-\eta} |S^*|}
 & \leq
  2p^{4|S^*|}
  \\
  & \leq p^{2|S^*|}
&
\text{by Property {\ref{prop:size}}.}
\end{align*}
This completes the proof of \eqref{eq:Z}.
\end{proof}

\medskip

To complete the proof of part 1 of Lemma \ref{lem:mainAdv},
it remains to prove Inequality \eqref{ineq:M1-improved-ExpZ}.

\begin{proof}[Proof of \eqref{ineq:M1-improved-ExpZ}]

For any $v \in S^*$, we define $\cU_v$ as the event $\left|N^-_{H^*} (v) \cap \widehat{\cF}_{\tX^*_T}\right| \leq \Delta^{1-3\eta}$.

Fix $v\in S^*$ where  $\cU_v$ holds. Then, $v$ has at most $\Delta^{1-3\eta}$  frozen down-neighbors.
In the worst-case these $\Delta^{1-3\eta}$ frozen down-neighbors
and the $|N_G(v) \setminus N_{H^*}(v)| \leq \D^{1-2\eta}$ adversary moves
reduce the number of available colors for $v$ by one each.  Now we will apply the Dyer-Frieze approach \cite{DyerFrieze}
  in $N^-_{H^*} (v) \setminus \widehat{\cF}_{\tX^*_T}$, as in the proof of \eqref{exp-HD}. Recall that by Property {\ref{prop:independence}} no vertex in $N^-_{H^*}(S^*)$ is a descendant of another vertex in $N^-_{H^*}(S^*)$, hence, the vertices in $N^-_{H^*}(v)$  receive independent colors in $\tX^*$.
We get that the expected number of available colors for $v$ in $\widehat{X}^*$ after recoloring $N_{H^*}^-(v)$ is at least $\approx ke^{-\D/k} - (\D^{1-3\eta} + \D^{1-2\eta}) \geq 2\D^{1-\eta/2}$.
That is,
\begin{equation*}\label{eq:expAva1}
\ExpCond{Z_v}{\tX^*_T} \geq 2\D^{1-\eta/2}  \ind{\cU_v},
\end{equation*}
and thus,
\begin{equation}\label{eq:expAva1aa}
 \ExpCond{Z}{\tX^*_T}
  \geq
   2\Delta^{1-\eta/2} \sum_{v\in S^*}\ind{\cU_v}.
\end{equation}
From the definition of $\cU_v$ we have
\begin{align*}
\sum_{v\in S^*}\ind{\neg \cU_v}
&\leq \sum_{v\in S^*} \frac{|N_{H^*}^-(v) \cap \widehat{\cF}_{\tX^*_T}|}{\Delta^{1-3\eta}}\\
&\leq \frac{30}{\Delta^{1-3\eta}}|N_{H^*}^-(S^*) \cap \widehat{\cF}_{\tX^*_T}| & \text{using property {\ref{prop:lipschitz}}}
\end{align*}
and thus
\[
\sum_{v\in S^*}\ind{\cU_v} \geq \frac 12|S^*|\ind{\cU}
\]
so, from \eqref{eq:expAva1aa} we have:
\begin{equation*}
 \ExpCond{Z}{\tX^*_T} \geq \Delta^{1-\eta/2}|S^*| \ind{\cU}
\end{equation*}
\end{proof}
Now we prove part 2 of Lemma \ref{lem:mainAdv}.
\begin{proof}[Proof of part 2 of Lemma \ref{lem:mainAdv}]
Let $S' =  S^* \setminus ( \widehat{\cF}_{\tX^*_{T'}} \cup \widehat{\cF}_{\tY^*_{T'}}\})$. Independently of the colors blocked by the adversary, for any $T' \leq t \leq T''$, all the vertices in $S'$ have at least $2\Delta^{1-\eta} - \Delta^{1-2\eta} > \Delta^{1-\eta}$ available colors in $\tX^*_t$ and in $\tY^*_t$ . Thus when a vertex in $S'$ is chosen, the probability of crossing it, independently of which two colors are forbidden by the adversary is at most $2/\Delta^{1-\eta}$. So for each $v \in S'$ if $v \in L_{\ell}$, there are  $T_{\ell} = |L_{\ell}|\ln \Delta$ trials where the adversary crosses $v$ with probability at most $2/\Delta^{1-\eta}|L_{\ell}|$ in each trial. Thus, conditioning on $\tX^*_{T'}$, for any $v \in S' \cap L_{\ell}$,
\[
\PrbCond{\cK(v)}{\tX^*_{T'}} \leq 1-(1-2/\Delta^{1-\eta}|L_{\ell}|)^{|L_{\ell}|\ln \Delta} \leq 1 - e^{-2\Delta^{-1+\eta}\ln \Delta} \leq \Delta^{-1+2\eta}.
\]
Now, to cross $S$ the adversary has to cross $S'$. Thus
\[
\PrbCond{\cK(S)}{\tX^*_{T'}} \leq \PrbCond{\cK(S')}{\tX^*_{T'}} = \PrbCond{\cap_{i=1}^{|S'|}\cK( s'_i)}{\tX^*_{T'}} \leq \Delta^{-(1-2\eta)|S'|}.
 \]
So we have
\begin{equation}\label{eq:sizeS'}
\PrbCond{\cK(S)}{\wX_{0,0},\wY_{0,0}} \leq \Delta^{-(1-2\eta)97|S|/100} + \PrbCond{|S'| < 97|S|/100}{\wX_{0,0},\wY_{0,0}}.
\end{equation}

To bound the second term of the RHS of \eqref{eq:sizeS'} we use that \eqref{eq:Z} holds for any subset of $S^*$. Thus,
\begin{align*}
\PrbCond{ |S^* \cap \widehat{\cF}_{\tX^*_{T'}}| \geq |S|/100}{\wX_{0,0}}
&\leq \sum_{U\subseteq S^*:|U| =  |S|/100} \PrbCond{ U \subseteq  \widehat{\cF}_{\tX^*_{T'}}}{\wX_{0,0}} \\
&\leq {|S^*| \choose |S|/100} p^{2|S|/100} \\
&\leq (99e)^{99}e^{-\Delta^{1/2} |S|/50} &\text{using  {\ref{prop:size}}}\\
& \leq \Delta^{-|S|}.
\end{align*}
And by symmetry,
\[
\PrbCond{ |S^* \cap \widehat{\cF}_{\tY^*_{T'}}| \geq |S^*|/100}{\wY_{0,0}} \leq \Delta^{-|S|}.
\]
Thus, using $|S \setminus S'| \leq |S \setminus S^*| + |S^* \cap \widehat{\cF}_{\tX^*_{T'}}| + |S^* \cap \widehat{\cF}_{\tY^*_{T'}}|$ and {\ref{prop:size}}, we have
\begin{align}
\nonumber
\lefteqn{
\hspace*{-.7in}
\PrbCond{|S'| \geq 3|S|/100}{\wX_{0,0},\wY_{0,0}}
}
\\
\nonumber
&\leq \PrbCond{|S^* \cap \widehat{\cF}_{\tX^*_{T'}}| \geq |S^*|/100 \text{ or } |S^* \cap \widehat{\cF}_{\tX^*_{T'}}| \geq |S^*|/100}{X_{0,0},Y_{0,0}}\\
\nonumber&\leq \PrbCond{ |S^* \cap \widehat{\cF}_{\tX^*_{T'}}| \geq |S^*|/100}{\wX_{0,0}} + \PrbCond{ |S^* \cap \widehat{\cF}_{\tY^*_{T'}}| \geq |S^*|/100}{\wY_{0,0}}\\
\label{eq:sizeS'compt}&\leq 2\Delta^{-|S|}.
\end{align}
Part 2 of Lemma \ref{lem:mainAdv} follows from \eqref{eq:sizeS'} and \eqref{eq:sizeS'compt}.
\end{proof}

\subsection{Proof of Lemma \ref{lemma:structural}: Structural Lemma}

To prove Lemma \ref{lemma:structural} we will delete edges from $H$ and drop some vertices from $S$ to obtain $H^*$ and $S^*$ with the desired properties. We will call edges in $G\setminus H$  ``adversarial".
For a vertex $v$, we call the (undirected) edge $(u,v)$ 
 a down-edge from $v$ if $u$ is a down-neighbor of $v$ (i.e., $u\in N^-(v)$). 
 If $u$ is an up-neighbor we call the edge an up-edge from $v$.
 
Let $D = \{v \in V:|N_G(v) - N_H(v)| > \Delta^{1-3\eta}\}$, the set of vertices which have ``many'' adversarial incident edges. 
Note, using part (\ref{hyp:small-adv}) of the hypothesis of Lemma
\ref{lemma:structural} we have that 
$D$ is small compared to $S$:
\begin{equation}\label{eq:sizeManyA}
|D|\leq \frac 1{\Delta^{1-3\eta}} \sum_{v\in V} |N_G(v) - N_H(v)| \leq \Delta^{-\eta }|S|.
\end{equation}
To obtain Property {\ref{prop:changes}}(a), we are going to avoid deleting any down-edge from $D$.
To obtain Property {\ref{prop:changes}}(b), we will bound the total number of deleted edges.

Let $B$ the set of vertices in $N^-_H(S)$ with too many up neighbors in $S$. The idea to obtain Property {\ref{prop:lipschitz}} is to ``drop" $B$ from $N_H^-(S)$. To do this we will try to delete all up-edges from $B$. This may not be possible without violating {\ref{prop:changes}}(a), as there might be vertices (e.g., those in $S$) with too many down-neighbors in $B$.  We call such vertices the heavy parents of $B$. We will show that the set of heavy parents of $B$ in $S$ is small compared to $S$, and thus we can drop those vertices from $S$, and delete any edge between a non-heavy parent of $B$ and $B$.

To obtain Property {\ref{prop:independence}}, we will eliminate up paths from $W = N^-_H(S)$ to itself. We try again deleting up-edges from $W$, and  
now the heavy parents of $W$
are an obstacle to do this (in this case every element of $S$ is a heavy parent
of $W$). We then define heavy ancestors as the closure under ``heavy parenting", and we eliminate all up-edges between $W$ union its heavy ancestors and the rest of $H$.

Now, we formalize the notion of heavy parents and ancestors and study their properties. Then we present the proof of Lemma \ref{lemma:structural}.

Given a set $U\subset V$ we say $v$ is a {\it heavy parent} for $U$ (in $H$) if $|N^-_H(v) \bigcap U| \geq \Delta^{5/6}$. We define the set of $\cH_H(U)$ of {\it heavy ancestors of $U$ (in $H$)} as the closure of $U$ under heavy parents. Namely, let
$U_0 = U$ and for any $i > 0 $, given $U_{<i} = \bigcup_{j <i} U_j$ let
\[U_i = \{v \in V: |N_H^-(v) \bigcap U_{<i}| > \Delta^{5/6} \} \bigcup U_{<i}.\]
Define
\[
\cH_H(U) = \bigcup_{i>0} U_i
\]

The set of heavy ancestors is no larger than $U$:
\begin{lemma}\label{lem:hAnc} Given a planar graph, $H = (V,E_H)$ and $U \subset V$,
\begin{enumerate}
\item For all $v \in V\setminus \cH_H(U)$, $|N^-_H(v) \bigcap (U \bigcup \cH_H(U))| \leq \Delta^{5/6}$
\item $|\cH_H(U)| \leq \D^{-2/3}|U|$
\end{enumerate}
\end{lemma}
\begin{proof} Part {\it 1} is straightforward from the definition. To prove Part {\it 2}, consider the graph induced in $H$ by $U \bigcup \cH_H(U)$. In this graph, for every $u \in \cH_H(U)$ we have $\deg(u) \geq \Delta^{5/6}$. As the average degree of a planar graph is $\leq 6$, $|\cH_H(U)|\Delta^{5/6} + |U \setminus \cH_H(U)| \leq 6(|\cH_H(U)| + |U|)$ and thus,
$
|\cH_H(U)| \leq \frac {5|U|}{\Delta^{5/6}-7}  \leq \Delta^{-2/3}|U|.
$
\end{proof}
Now we give a general procedure to eliminate up-paths between sets:
Given $U_1,U_2 \subset V$, not necessarily disjoint, let $H'(U_1,U_2)$ be the graph obtained by deleting from $H$ all up-edges from $U_1 \bigcup \cH_H(U_1)$ to $V \setminus (\cH_H(U_1) \bigcup (U_1 \setminus U_2))$ that are in an up-path from $U_1$ to $U_2$.
\begin{lemma}\label{lem:notUpPath}Given $H = (V,E)$ and $U_1,U_2 \subset V,$ let $H'= H'(U_1,U_2)$.
\begin{enumerate}
\item In $H'$ there are no up-paths from $U_1$ to $U_2 \setminus \cH_H(U_1)$.
\item The number of deleted edges to create $H'$ is $\leq 2\Delta^{5/6}|U_1|$, i.e.,  $|E \setminus E'| \leq 2\Delta^{5/6}|U_1|$ where $E'$ are the edges of $H'$.
\item For all $v \in V$, $\deg^-_{H'}(v) \geq \deg^-_H(v) - \Delta^{5/6}$.
\item For all $v$ not in an up-path from $U_1$ to $U_2$, $\deg^-_{H'}(v) = \deg^-_H(v)$.
\item For all $v \in U_1 \setminus U_2$,  $\deg^-_{H'}(v) = \deg^-_H(v)$.
\end{enumerate}
\end{lemma}
\begin{proof} To prove Part {\it 1}, let $\sigma = v_0,v_1,\dots,v_{m}$ where $m\geq 1$ be  an up-path from $U_1$ to $U_2 \setminus \cH_H(U_1)$ in $H$.
 Let $j$ be the minimum $i>0$ such that $v_i \notin \cH_H(U_1)\bigcup (U_1 \setminus U_2)$. Such a $j$ exists because $v_m \in U_2 \setminus \cH_H(U_1) \subseteq V \setminus (\cH_H(U_1)\bigcup (U_1 \setminus U_2))$. We have that $(v_{j-1},v_j)$ is an up-edge in $G$, and $v_j \in  V \setminus (\cH_H(U_1) \bigcup (U_1 \setminus U_2))$. If $j=1$, then $v_{j-1} = v_0 \in U_1$. If $j >1$ then $v_{j-1} \in \cH_H(U_1)\bigcup (U_1 \setminus U_2) \subseteq  \cH_H(U_1) \bigcup U_1$, and by construction $(v_{j-1},v_j)$ has been deleted in $H'$ and thus $\sigma$ is not contained in $H'$.

From Lemma \ref{lem:hAnc}-1, Part {\it 3} follows. Parts {\it 4} and {\it 5} follow directly from the construction of $H'$.  To prove Part {\it 2}, notice that all deleted edges are up-edges for $U_1 \bigcup \cH_H(U_1)$, and thus $|E \setminus E'| \leq \Delta^{5/6}|U_1 \bigcup \cH_H(U_1)| \leq 2\Delta^{5/6}|U_1|$, using Lemma \ref{lem:hAnc}-2 to obtain $|\cH_H(U_1)| \leq |U_1|$.
\end{proof}

We are ready to prove Lemma \ref{lemma:structural} using Lemmas \ref{lem:hAnc} and \ref{lem:notUpPath}

\begin{proof}[Proof of Lemma \ref{lemma:structural}]
First, we eliminate any up-path from $D$ to $S$. Let $H_1 = H'(D,S\setminus D)$. Let $S_1 = S \setminus (D \bigcup \cH_H(D))$.
From Lemma \ref{lem:hAnc} and \eqref{eq:sizeManyA},
\begin{equation}
\label{eq:sizeS1}
|S_1| \geq |S|  - ( |D| + |\cH_H(D)|) \geq |S| - (1 + \Delta^{-2/3}) |D| \geq (1- 2\Delta^{-\eta})|S|
\end{equation}
From parts {\it 1} and {\it 2} of Lemma \ref{lem:notUpPath} we have
\begin{corollary}\label{corollary:noupD-S1}
\begin{enumerate}
\item In $H_1$ there is no up-path from $D$ to $S_1$.
\item $|E \setminus E_1| \leq 2\Delta^{5/6}|D| \leq 2\Delta^{5/6}|S|$.
\end{enumerate}
\end{corollary}

\noindent 
 Let $B = \{u \in N^-_{H_1}(S_1):|N^+_{H_1}(u) \bigcap S_1| > 30\}$. From planarity, the average degree of any subgraph of $H$ is $6$. Thus looking at the subgraph induced by $B$ and $S_1$ we have $30|B| \leq 6(|B| + |S_1|)$ and therefore $|B| \leq |S_1|/4$.
Let $H_2 = H_1'(B,S_1)$ and $S_2 = S_1 \setminus \cH_{H_1}(B)$.
From Lemma \ref{lem:hAnc} and \eqref{eq:sizeS1},
\begin{equation}
\label{eq:sizeS2}
|S_2| \geq |S_1|  -  |\cH_{H_1}(B)| \geq |S_1| -  \Delta^{-2/3}|B| \geq (1- \Delta^{-2/3}/4)|S_1| \geq (1- 3\Delta^{-\eta})|S|
\end{equation}
From parts {\it 1} and {\it 2} of Lemma \ref{lem:notUpPath} we have
\begin{corollary}\label{corollary:noupB-S2}
\begin{enumerate}
\item In $H_2$ there are no up-paths from $B$ to $S_2$
\item $|E_1 \setminus E_2| \leq 2\Delta^{5/6}|B| \leq \Delta^{5/6}|S|/2$.
\end{enumerate}
\end{corollary}

Notice that from Part 1 of Corollary \ref{corollary:noupB-S2}, 
we have that for every $u \in N^-_{H_2}(S_2)$, $u\notin B$.
Thus, by the definition of the set $B$, 
$|N^+_{H_2}(u) \bigcap S_2| \leq 30$.  This will 
imply Property {\ref{prop:lipschitz}} in the final graph we construct.

\noindent Now we eliminate all up-paths from  $N^-_{H_2}(S_2)$ to itself. Let $W_2 = N^-_{H_2}(S_2)$. Using Lemma \ref{lem:notUpPath} we can eliminate all up-paths from $W_2$ to $W_2 \setminus \cH_{H_2}(W_2)$. Thus, we will first drop $W_2 \bigcap \cH_{H_2}(W_2)$ from $N^-_{H_2}(S_2)$.

Let $G_3 = G_2'(\cH_{H_2}(W),S_2)$,  and $S_3 = S_2 \setminus \cH_{H_2}(\cH_{H_2}(W_2))$.
From Lemma \ref{lem:hAnc} and \eqref{eq:sizeS2},
\begin{equation}
\label{eq:sizeS4}
|S_3| \geq |S_2|  -  |\cH_{H_2}(\cH_{H_2}(W_2))| \geq |S_2| - \Delta^{-4/3}|W_2| \geq (1- \Delta^{-1/3})|S_2| \geq (1- 4\Delta^{-\eta})|S|
\end{equation}
From parts {\it 1} and {\it 2} of Lemma \ref{lem:notUpPath} we have
\begin{corollary}\label{corollary:noupW-WcHW}
\begin{enumerate}
\item In $H_3$ there are no up-paths from $\cH_{H_2}(W_2)$ to $S_3$.
\item $|E_2 \setminus E_3| \leq 2\Delta^{5/6}|\cH_{H_2}(W_2)| \leq 2\Delta^{1/6}|W_2| \leq 2 \Delta^{7/6}|S_2| \leq 2 \Delta^{7/6}|S|$.
\end{enumerate}
\end{corollary}

\noindent 
Let $W_3 =  N^-_{H_3}(S_3)$.  
Note, we have $W_3 \subseteq W_2$.
Therefore, we have that:
\[
\cH_{H_3}(W_3) \bigcap W_3 \subseteq \cH_{H_2}(W_3) \bigcap W_3 \subseteq \cH_{H_2}(W_2) \bigcap W_3.
 \]
But, from Part 1 of Corollary \ref{corollary:noupW-WcHW},  \[ \cH_{H_2}(W_2)~\bigcap~W_3 =\cH_{H_2}(W_2)~\bigcap~N^-_{H_3}(S_3)= \emptyset. \]
 And thus, 
 \begin{equation}
 \label{eq:W4}
 W_3 \setminus \cH_{H_3}(W_3) = W_3.
 \end{equation}
 Let $H^* = H_3'(W_3,W_3)$ and $S^* = S_3$.
From parts {\it 1} and {\it 2} of Lemma \ref{lem:notUpPath} and \eqref{eq:W4} we have:
\begin{corollary}\label{corollary:noupW-W}
\begin{enumerate}
\item In $H^*$ there are no up-paths between pairs of vertices in $W_3$.
\item $|E_3 \setminus E^*| \leq 2\Delta^{5/6}|W_3| \leq 2\Delta^{11/6} |S^*| $.
\end{enumerate}
\end{corollary}

We now check that $H^*$ and $S^*$ satisfy the properties stated in 
Lemma \ref{lemma:structural}.
We begin with Property \ref{prop:size}.
From \eqref{eq:sizeS4},
\begin{equation}\label{eq:sizeS*}
|S^*| = |S_3| \geq (1- 4\Delta^{-\eta}) |S| \geq 99|S|/100.
\end{equation}

Now we prove Property \ref{prop:changes} holds.
For all $v\in V$,
\begin{align}
\nonumber
|N_G(v) - N_{H^*}(v)| + |N^+_{H^*}(v)|
&= |N_G(v) - N_{H}(v)| + |N_H(v) - N_{H^*}(v)| + |N^+_{H^*}(v)|\\
\label{eq:proveP2}
&= |N_G(v) - N_{H}(v)| + |N^+_{H}(v)| + |N^-_H(v) - N^-_{H^*}(v)|
\end{align}
By part {\it 3} of Lemma \ref{lem:notUpPath}, for any $v \in V$, 
\begin{align*}\nonumber
\deg^-_{H^*} (v)
&\geq \deg^-_{H_3} (v) - \Delta^{5/6}
\geq \deg^-_{H_2} (v) - 2\Delta^{5/6}
\geq \deg^-_{H_1} (v) - 3\Delta^{5/6}
\geq \deg^-_{H} (v) - 4\Delta^{5/6}.
\end{align*}
Thus, if $v \notin D$, using \eqref{eq:proveP2},
\[
|N_G(v) - N_{H^*}(v)| + |N^+_{H^*}(v)|
\leq \Delta^{1-3\eta} + \Delta^{5/6} + 4\Delta^{5/6}
\leq \Delta^{1-2\eta},
\]
Also, for any $v \in D$, by Part {\it 1} of Corollary \ref{corollary:noupD-S1} we can apply Part {\it 4} of Lemma \ref{lem:notUpPath} to an appropriate sequence of graphs to obtain
\[
\deg^-_{H^*}(v) = \deg^-_{H_3}(v) = \deg^-_{H_2}(v)= \deg^-_{H_1}(v).
\]
By Lemma \ref{lem:notUpPath}, Part {\it 5}
\[
\deg^-_{H_1}(v) = \deg^-_{H}(v)
\]
Finally,  using \eqref{eq:proveP2}
\[
|N_G(v) - N_{H^*}(v)| + |N^+_{H^*}(v)|
= |N_G(v) - N_{H}(v)| + |N^+_{H}(v)|
\leq \Delta^{1-2\eta}.
\]
This proves Part (a) of Property \ref{prop:changes}.
For Part (b), we have that:
\begin{align*}
\lefteqn{
\sum_{v\in V} |N_G(v) - N_{H^*}(v)|
}
\\
&= \sum_{v\in V} (|N_G(v) - N_{H}(v)|  + |N_H(v) - N_{H^*}(v)|) \\
&= \sum_{v\in V} |N_G(v) - N_{H}(v)|  +  |E - E^*| \\
&= \sum_{v \in V} |N_G(v) - N_{H}(v)| + |E- E_1| + |E_1 - E_2| + |E_2 - E_3| + |E_3 - E^*| \hspace*{-4in}\\
&\leq \Delta^{1-4\eta}|S| + (3\Delta^{5/6} + 2\Delta^{7/6})|S| + 2\Delta^{11/6} |S^*|
    & \text{(by Part {\it 2} of Corollaries \ref{corollary:noupD-S1}-\ref{corollary:noupW-W})}\\
&\leq 3\Delta^{11/6}|S^*|
    & \text{(by \eqref{eq:sizeS*})}\\
&\leq \Delta^{2-8\eta}|(S^*)|.
\end{align*}
This proves part (b) of Property {\ref{prop:changes}}.
For Property \ref{prop:lipschitz}, 
as noted earlier, 
for all $u\in N^-_{H^*}(S^*)$, $u \in N^-_{H_2}(S_2)$ and from Part {\it 1} of Corollary \ref{corollary:noupB-S2},
\[
\left|N^+_{H^*}(v) \bigcap S^* \right| \leq |N^+_{H_2}(u) \bigcap S_2| \leq 30.
\]

Finally for Property \ref{prop:independence},
let $v,w\in S^*$, and let $y\in N^-_{H^*}(v)$ and $z\in N^-_{H^*}(w)$. Then $y,z \in W_3$ and from Part {\it 1} of Corollary \ref{corollary:noupW-W}, 
there is no up-path  from $y$ to $z$ in $H^*$.
This proves \ref{prop:independence} and completes the proof of Lemma
\ref{lemma:structural}.

\end{proof}

\section{Comparison Argument for the Glauber Dynamics}
\label{sec:comparison}

In this section we prove Part (ii) of Theorem \ref{thm:constant-degree}
using Part (i) of that same theorem and the comparison technique introduced by
Diaconis and Saloff-Coste \cite{DSc}.
The comparison result we prove is closely related to that of
Dyer et al \cite[Theorem 32]{DGJ},
which proves that the inverse of the spectral gap of the Glauber dynamics is at most a factor $O(n^2k)$
worse than that of systematic scan.  A preliminary version of this paper
claimed that their argument generalizes to the level-set dynamics with the same bounds as for
systematic scan.  Linji Yang (personal communication) pointed out that
a straightforward application of their proof for the level-set dynamics
adds an extra factor of $O(n)$
due to the number of times a vertex may be recolored during one ``scan'' of the level-set dynamics.
He suggested the following proof which uses ideas of Sinclair \cite[Proof of Theorem 8]{Sinclair}.

\newcommand{\PGl}{P_{\mathrm{Gl}}}
\newcommand{\TmixGl}{T_{\mathrm{Gl}}}
\newcommand{\PLS}{P_{\mathrm{LS}}}
\newcommand{\TmixLS}{T_{\mathrm{LS}}}
%

Let $\PGl$ denote the transition matrix for the Glauber dynamics,
and let $\TmixGl$ denote its mixing time.
Let $\PLS$ denote the transition matrix of the level-set dynamics
where one transition does $H$ rounds of the dynamics.
Thus, in one transition, all of the levels are updated, and so
one transition of $\PLS$ corresponds to one ``scan'' by the dynamics.
Let $\pi$ denote the stationary distribution of the two chains, namely
the uniform distribution over the set $\Omega$ of $k$-colorings of $G$.
Finally, let $\TmixLS$ denote the mixing time of $\PLS$.
We have proven that:
\[
\TmixLS \leq O(\log{n}).
\]

Let $\sigma, \sigma'\in\Omega$ be a pair of colorings where
$\PLS(\sigma,\sigma')>0$.
For $T=n\log{\Delta}$, let
$\alpha=(\alpha_1,\dots,\alpha_{T})$ denote a sequence of $T$ vertices
and $\beta=(\beta_1,\dots,\beta_T)$
denote a sequence of $T$ colors.
We say that a transition $\sigma\rightarrow\sigma'$
of the level-set dynamics has update sequence
$(\alpha,\beta)$ if for $i=1\rightarrow T$ the level-set dynamics
 at step $i$, updates vertex $\alpha_i$ with color $\beta_i$.
 Let $\PLS(\sigma,\sigma',\alpha,\beta)$ denote the probability that
the level-set dynamics transitions from $\sigma$ to $\sigma'$ with
update sequence $(\alpha,\beta)$.

For each $\sigma,\sigma',\alpha,\beta$
where $\PLS(\sigma,\sigma',\alpha,\beta)>0$ let
$\gamma_{\sigma,\sigma'}(\alpha,\beta)$ denote the path of $n\log{\Delta}$
Glauber transitions defined by the update sequence $(\alpha,\beta)$,
where any cycles are removed so that we are left with
a simple path along Glauber transitions.
To be clear, for two different update sequences $(\alpha,\beta)$ and
$(\alpha',\beta')$ we may have $\gamma_{\sigma,\sigma'}(\alpha,\beta) =
\gamma_{\sigma,\sigma'}(\alpha',\beta')$.
Let
\[  f(\gamma_{\sigma,\sigma'}(\alpha,\beta)) =
\PLS(\sigma,\sigma',\alpha,\beta)
\]
Thus, $\sum_{\alpha,\beta} f(\gamma_{\sigma,\sigma'}(\alpha,\beta))  =
\PLS(\sigma,\sigma')$ and hence $f$
defines a valid flow as in \cite{DSc}.

In the comparison technique of \cite{DSc}
we need to bound the congestion $A$ defined as, the maximum
over colorings $\tau,\tau'$ where $\PGl(\tau,\tau')>0$ of the
following quantity referred to as the congestion of the flow:
\begin{eqnarray*}
A
&=& \frac{1}{\pi(\tau)\PGl(\tau,\tau')}\sum_{\sigma,\sigma',\alpha,\beta:
\atop \gamma=\gamma_{\sigma,\sigma'}(\alpha,\beta)\ni \tau\rightarrow\tau'}
\pi(\sigma)f(\gamma)|\gamma|
\\
&\le& \frac{(n\Aval_{\tau}(v))(n\log{\Delta})}{\pi(\tau)}\sum_{\sigma,\sigma',\alpha,\beta:
\atop \gamma=\gamma_{\sigma,\sigma'}(\alpha,\beta)\ni \tau\rightarrow\tau'}
\pi(\sigma)\PLS(\sigma,\sigma',\alpha,\beta)
\end{eqnarray*}
We write, 
\begin{eqnarray*}
\lefteqn{
\sum_{\sigma,\sigma',\pi,\eta:
\atop \gamma\ni \tau\rightarrow\tau'}
\pi(\sigma)\PLS(\sigma,\sigma',\alpha,\beta)
}
\\
& = & \sum_{\sigma} \pi(\sigma)
\Prb{\mbox{Starting from $\sigma$, that
$\tau\rightarrow\tau'$ is traversed during one transition
of $\PLS$}}
\\
& = &
\Prb{\mbox{Starting from $\sigma\sim\pi$, that
$\tau\rightarrow\tau'$ is traversed
during one transition
of $\PLS$}}
\end{eqnarray*}

For $X_0=\sigma$, let  $T_i$ denote the time when
the level-set dynamics begins recoloring level $i$.  Let $j$
denote the level of $v$ the vertex recolored during the
Glauber transition $\tau\rightarrow\tau'$.
Since $X_0=\sigma\sim\pi$ then for all $t\geq 0$,
$X_t\sim\pi$.  For $Z_t\sim\pi$, if $Z_{t+1}$ is defined
by the Glauber dynamics, the probability that
$(Z_t\rightarrow Z_{t+1})=(\tau\rightarrow\tau')$
is $\frac{\pi(\tau)}{n\Aval_{\tau}(v)}$.
Similarly, for $X_t$ defined by the level-set dynamics, for $t$ where
$T_j\le t<T_{j+1}$, the probability that
$(X_t\rightarrow X_{t+1})=(\tau\rightarrow\tau')$
is $\frac{\pi(\tau)}{|L_j|\Aval_{\tau}(v)}$.
Since there are $|L_j|\log{\Delta}$ such times $t$, we have that:
\[
A \leq n^2\log^2{\Delta}
\]
Part (ii) of Theorem \ref{thm:constant-degree} now follows
from Theorem 2.3 of \cite{DSc} together with standard results
relating the spectral gap to the mixing time (c.f., \cite[Proposition 1]{Sinclair}).

\section{Concluding Remarks}

In an earlier version of this work we asked whether
the mixing time is super-polynomial for
the Glauber dynamics for the complete $(\Delta-1)$-ary tree
with $k = 3$, when $\Delta = O(1)$?
This was resolved recently by Lucier and Molloy \cite{Lucier-Molloy} and
Goldberg et al \cite{GJK} who showed that for constant $\Delta$ and constant $k$,
the mixing time is polynomial.  (See also \cite{TVVY} for further improvements
regarding the mixing time of the Glauber dynamics on the complete tree.)

An intriguing direction is proving polynomial mixing time of the Glauber dynamics
for planar graphs with $k<<\Delta$ for constant $k$ and $\Delta$.
Another interesting direction is proving rapid mixing of the Glauber dynamics
for general bipartite graphs.
It is even possible that there are efficient sampling algorithms for triangle-free
graphs when $k<\Delta$ since
Johansson \cite{triangle-free,MolloyReed} has shown that the chromatic
number of such graphs is $O(\Delta/\log{\Delta})$.

\end{document}